\documentclass{amsart}
\def\data{30 September 2017}

\usepackage{amssymb}
\usepackage{amsfonts}
\usepackage{amsmath}
\usepackage{amsthm}
\usepackage{url}
\usepackage{color}
\usepackage[small,nohug,heads=LaTeX]{diagrams}
\diagramstyle[labelstyle=\scriptstyle]
\usepackage{graphicx}
\usepackage{mathabx}

\newcommand{\bq}{\begin{quote}}
\newcommand{\eq}{\end{quote}}
\newcommand{\bi}{\begin{itemize}}
\newcommand{\ei}{\end{itemize}}
\newcommand{\bd}{\begin{description}}
\newcommand{\ed}{\end{description}}
\newcommand{\ben}{\begin{enumerate}}
\newcommand{\een}{\end{enumerate}}
\newcommand{\bbm}{\begin{bmatrix}}
\newcommand{\ebm}{\end{bmatrix}}
\newcommand{\bea}{\begin{eqnarray*}}
\newcommand{\eea}{\end{eqnarray*}}

\newtheorem{theorem}{Theorem}[section]

\newtheorem{lemma}[theorem]{Lemma}

\newtheorem{corollary}[theorem]{Corollary}

\newcommand{\bx}[1]{\mathsf{#1}}     
\newcommand{\bxg}[1]{\boldsymbol #1} 

\def\bC{\mathbb{C}}
\def\bF{\mathbb{F}}

\def\bP{\mathbb{P}}

\def\bR{\mathbb{R}}

\def\Ii{\mathfrak{I}}
\def\Pp{\mathfrak{P}}
\def\Qq{\mathfrak{Q}}

\def\lie{\mathfrak{l}}

\def\lsl{\mathfrak{sl}}
\def\tor{\mathfrak{t}}
\def\lia{\mathfrak{a}}

\def\AA{\mathsf{A}}
\def\BB{\mathsf{B}}
\def\CC{\mathsf{C}}
\def\EE{\mathsf{E}}
\def\FF{\mathsf{F}}

\def\HH{\mathsf{H}}
\def\KK{\mathsf{K}}

\def\NN{\mathsf{N}}

\def\SS{\mathsf{S}}
\def\TT{\mathsf{T}}
\def\UU{\mathsf{U}}
\def\VV{\mathsf{V}}
\def\WW{\mathsf{W}}
\def\XX{\mathsf{X}}
\def\YY{\mathsf{Y}}
\def\ZZ{\mathsf{Z}}

\newcommand{\jj}{\mathsf{j}}
\def\kk{\mathbf{k}}
\def\ll{\mathbf{l}}
\def\uu{\mathbf{u}}
\def\polar{\boldsymbol{\pi}}

\def\Tr{\mathop{{\rm Tr}}}

\def\Lie{\mathop{{\rm Lie}}}

\def\Alt{{\rm Alt}}
\def\Sym{{\rm Sym}}
\def\2G2{\ensuremath{^2{\rm G}_2}}
\def\sl{{\rm{SL}}}
\def\gl{{\rm{GL}}}

\def\psl{{\rm{PSL}}}
\def\pgl{{\rm{PGL}}}

\def\so{{\rm{SO}}}

\def\qpone{q \equiv 1  \bmod 4}
\def\qmone{q \equiv -1 \bmod 4}
\def\fmone{|\bF| \equiv -1 \bmod 4}
\def\fpone{|\bF| \equiv 1 \bmod 4}

\AtBeginDocument{%
  \def\MR#1{}
  }

\begin{document}

\title[Black box groups $\psl_2(\bF_q)$]{Adjoint representations of black box groups $\psl_2(\bF_q)$}
\date{\data}

\author{Alexandre Borovik}
\address{School of Mathematics, University of Manchester, UK; alexandre@borovik.net}
\author{\c{S}\"{u}kr\"{u} Yal\c{c}\i nkaya}
\address{Department of Mathematics, Istanbul University, Turkey; sukru.yalcinkaya@istanbul.edu.tr}


\begin{abstract}

Given a black box group $\YY$ encrypting $\psl_2(\bF)$ over an unknown field $\bF$ of unknown odd characteristic $p$ and a global exponent $E$ for $\YY$ (that is, an integer $E$ such that $\bx{y}^E=1$ for all $\bx{y} \in \YY$), we present a Las Vegas algorithm which constructs a unipotent element in $\YY$. The running time of our algorithm is polynomial in $\log E$. This answers the question posed by Babai and Beals in 1999. We also find the characteristic of the underlying field in time polynomial in $\log E$ and linear in $p$.

Furthermore, we construct, in probabilistic time polynomial in $\log E$,
\begin{itemize}
\item a black box group $\XX$ encrypting $\pgl_2(\bF) \cong\so_3(\bF)$, its subgroup $\YY^\circ$ of index $2$ isomorphic to $\YY$ and a probabilistic polynomial in $\log E$ time isomorphism $\YY^\circ \longrightarrow \YY$;
\item a black box field $\KK$,  and
\item polynomial time, in $\log E$, isomorphisms
\[
 \so_3(\KK) \longrightarrow \XX \longrightarrow \so_3(\KK).
\]
\end{itemize}

If, in addition, we know $p$ and the standard explicitly given finite field $\bF_q$  isomorphic to the field $\bF$ then we construct, in time polynomial in $\log E$, isomorphism
\[
\so_3(\bF_q)\longrightarrow  \so_3(\KK).
\]

Unlike many papers on black box groups, our algorithms make no use of additional oracles other than the black box group operations. Moreover, our result acts as an $\sl_2$-oracle in the black box group theory.

We implemented our algorithms in GAP and tested them for groups such as $\psl_2(\bF)$ for $|\bF|=115756986668303657898962467957$ (a prime number).
\end{abstract}

\maketitle

\section{Introduction}
\subsection{The principal results}

Black box groups were introduced by Babai and Szemer\'{e}di \cite{babai84.229} as an idealized setting for randomized algorithms for solving permutation and matrix group problems in computational group theory. A black box group is a black box (or an oracle, or a device, or an algorithm) operating with $0$--$1$ strings of uniform length which encrypt (not necessarily in a unique way) elements of some finite group. In various classes of black box problems the isomorphism type of the encrypted group could be known in advance or unknown.

We denote a black box group encrypting a group $X$ by using the same letter in sans serif,
$\XX$,  and we apply the same convention to the strings $\bx{x}$ produced by $\XX$ which correspond to the group element $x \in X$.

All black box groups in this paper are assumed to satisfy Axioms BB1--BB4 stated in Sections \ref{sec:BB1--BB3} and \ref{sec:BB4}.
In particular, we assume that for every black box group $\XX$ we are given a global exponent, that is, an integer $E$ such that $x^E=1$ for all $x \in X$, and the computation of $\bx{x}^E$ is feasible.

In this paper, we present an algorithm which solves the old problem by Babai and Beals \cite[Problem~10.1]{babai99.30} that remained open since 1999. We prove the following theorem.

\begin{theorem}\label{cor:unipotent} Let $\YY$ be a black box group  encrypting $Y=\psl_2(\bF)$, where $\bF$ is an unknown finite field of unknown odd characteristic $p$ and let $E$ be a global exponent for $\YY$. Assume also that $|\bF|\geq 7$. Then there exists a Las Vegas algorithm which constructs a string representing a non-trivial unipotent element from $Y$in time polynomial in $\log E$. In particular, the characteristic $p$ of the underlying field can be found  in time  polynomial in $\log E$ and linear in $p$.
\end{theorem}

For a discussion of randomized algorithms and, in particular,  Las Vegas algorithms, see Section~\ref{sec:Monte-Carlo}.

In case of $p=2$, the Babai-Beals problem has been solved by Kantor and Kassabov \cite{kantor15.16}, and in Section~\ref{sec:even-case} we briefly discuss how our methods can also be applied to this case.

We exclude small cases $|\bF| < 7$ from consideration in this paper because they -- and, more generally, black box recognition of classical matrix groups over small fields are comprehensively treated in the memoir by Kantor and Seress \cite{kantor01.168}.

Note that, in Theorem \ref{cor:unipotent}, we do not have any information about the ground field of the group $Y$. However, we use some form of an upper bound on the size of this field which is implicitly present in the global exponent $E$. Note also that we construct a unipotent element without knowing its order and our algorithm is Las Vegas due to the unipotency test, Lemma~\ref{lem:test:uni}. To find the characteristic of the underlying field, we construct a unipotent element and then find its order.

In the special case when our black box group $\YY$ is explicitly represented by matrices, Theorem~\ref{cor:unipotent} takes the form that also remained unknown until now.

\begin{corollary}\label{cor:matrix}
Let $p$ be an odd prime number.
Given matrices $g_1, \dots, g_m$ in   the group $\gl_n(\bF_{p^k})$ of invertible matrices over a finite field\/ $\bF_{p^k}$ of odd characteristic $p$ which generate subgroup $G$ isomorphic to $\sl_2(\bF_{p^l})$, $p^l\geq 7$, we can find in $G$ a non-trivial unipotent element in probabilistic time  polynomial in $k,l,m,n$ and\/ $\log p$.
Our algorithm is Las Vegas.
\end{corollary}

Our next result is a solution to the problem of recognizing a black box group encrypting the group $\psl_2(\bF)$ defined over an unknown field of unknown odd characteristic.

\begin{theorem}\label{SL2-SO3}
Let $\YY$ be a black box group encrypting $\psl_2(\bF)$, where $\bF$ is an unknown field of unknown odd characteristic $p$ and let $E$ be a global exponent for $\YY$. Assume also that $|\bF|\geq 7$. Then we construct, in probabilistic time polynomial in $\log E$,

\begin{itemize}
\item[(a)] a black box group $\XX$ encrypting the group $\so_3(\bF)$, its subgroup $\YY^\circ$ of index $2$ isomorphic to $\YY$ and a probabilistic polynomial in $\log E$ time isomorphism $\YY^\circ \longrightarrow \YY$;
\item[(b)] a black box field $\KK$, and
\item[(c)]   polynomial time, in $\log E$, isomorphisms
\[
\so_3(\KK) \longrightarrow \XX \longrightarrow \so_3(\KK).
\]
\end{itemize}
Our algorithms are Las Vegas. If, in addition, we know $p$ and the standard explicitly given finite field $\bF_q$ of characteristic $p$ isomorphic to $\bF$ then we  construct, in $\log E$-time,  an isomorphism
\[
\so_3(\bF_q)\longrightarrow  \so_3(\KK).
\]
\end{theorem}

Since, by Theorem~\ref{cor:unipotent}, we can find the characteristic $p$ of the underlying field in time linear in $p$ and polynomial in $\log E$,
and also because finding efficient two-way isomorphisms between  a black box field of order $p^n$ and an explicitly given standard field can be done in time linear in $p$ and polynomial in time $n\log p$ (see Section \ref{sec:bbfields}),
we have a stronger result:

\begin{corollary}
Let $\XX$ be a black box group encrypting $\so_3(\bF)$, where $\bF$ is an unknown field of unknown odd characteristic $p$ and let $E$ be a global exponent for $\XX$. Assume also that $|\bF|\geq 7$. Then we construct, in time linear in $p$ and polynomial in\/ $\log E$, an isomorphism
\[
\XX\longleftrightarrow \so_3(\bF_q),
\]
where $\bF_q$ is the standard explicitly given finite field isomorphic to $\bF$.
\end{corollary}

We note here that our algorithm fully replaces the so-called ``\emph{$\sl_2$-oracle}'', an assumption of existence of two-way polynomial-time isomorphism between arbitrary black box group encrypting $\sl_2(\bF_{p^k})$ and the group $\sl_2(\bF_{p^k})$ over the standard explicitly given field $\bF_{p^k}$. Moreover, we do not use discrete logarithm oracle on finite fields in our algorithms as opposed to {majority of the existing algorithms for black box groups. (Given a generator $x$ of $\bF^*$ and a random element $y \in \bF$, a \emph{discrete logarithm oracle} finds an integer $k$ which satisfies $x^k=y$.)  The first use of an ``$\sl_2$-oracle'' appeared in 2001;
quite a number of papers referring to $\sl_2$-oracle and discrete logarithm oracle followed \cite{brooksbank03.162,brooksbank08.885,brooksbank01.95,brooksbank06.256}. The present paper together with \cite{BY2010, BY2013C} shows the way to eliminate discrete logarithm oracle and $\sl_2$-oracle entirely from black box recognition problems for classical groups of odd characteristic.

By replacing the axiom BB4 with BB5 (stated in Subsection \ref{sec:Euclid-Lobachevsky}), that is, removing the assumption of knowing a global exponent $E$ for the black box groups and assuming to have a function which computes square roots of group elements, when exist, we have the following more general result. It explains, in particular, why the characteristic of the field is not used in our Theorems \ref{cor:unipotent}  and \ref{SL2-SO3}.

\begin{theorem} $\YY$ be a black box group which satisfies axioms \emph{BB1--BB3} and encrypts the group $\psl_2(\bF)$  over some unknown finite field  $\bF$ of unknown odd characteristic $p$ with $|\bF|\geq 7$. Assume also that Axiom \emph{BB5} holds in $\YY$. Then we can construct, in probabilistic time polynomial in $\log E$,
\begin{itemize}
\item[(a)] a black box group $\XX$ encrypting the group $\so_3(\bF)$, its subgroup $\YY^\circ$ of index $2$ isomorphic to $\YY$ and a probabilistic polynomial in $\log E$ time isomorphism $\YY^\circ \longrightarrow \YY$;
\item[(b)] a black box field $\KK$, and
\item[(c)]   polynomial time, in $\log E$, isomorphisms
\[
\so_3(\KK) \longrightarrow \XX \longrightarrow \so_3(\KK).
\]
\end{itemize}
Our algorithms are Las Vegas. \label{th:square-roots}
\end{theorem}

We record this result here, but its discussion will be published by us elsewhere; it is linked to work by Ahmavaara \cite{YA1, YA2, YA34},  Kustaanheimo \cite{KH}, and other theoretical physicists who attempted to build a model of quantum mechanics based on a very big prime field. Recently, this theory was revisited by Zilber \cite{Zilber2012} who used a model-theoretic approach to the idea of ``quantization by looking at everything as if it  happens in a finite field''.  Our Theorem~\ref{th:square-roots} appears to fit well into an emerging theory linking model theory with physics: in view of \cite[Proposition 5.2]{Zilber2012}, finite groups $\pgl_2(\bF)$ \emph{structurally approximate}, in some explicitly defined sense,  the Minkowski group $\psl_2(\bC)$.

\subsection{A very brief outline of the proof}

The proof of Theorem~\ref{SL2-SO3} will be achieved as a sequence of steps some of which are interesting on their own.
 \bi
  \item[(a)] We construct a new black box group $\XX$ encrypting $\so_3(\bF)$, a group which contains $\psl_2(\bF)$ as a subgroup $\YY^\circ$ of index $2$, and map
  \[
\YY^\circ \hookrightarrow \YY,
\]
see Theorem~\ref{psl2-pgl2}.
\item[(b)] Using involutions in $\XX$, we construct a black box projective plane $\Pp$ that encrypts the projective plane of the $3$-dimensional space of the adjoint representation of $\pgl_2(\bF)\simeq \so_3(\bF)$ on its Lie algebra $\lie = \lsl_2(\bF)$. We describe how to produce random points in $\Pp$, describe lines through two points, construct intersection of two lines, etc. We shall note here that we do not list the elements in $\Pp$ or in any line in $\Pp$. Furthermore, the plane $\Pp$ has polarity induced by the Killing form on its underlying Lie algebra. The Lie algebra product induces on $\Pp$ a partial binary operation which we are able to compute using black box methods; we denote this operation by $\boxtimes$, call \emph{cross product} and systematically use in algorithms developed in the paper.
\item[(c)] We introduce a set of tools which allows us to coordinatize $\Pp$ by homogeneous coordinates over a black box field $\KK$ constructed in the projective plane $\Pp$. To that end, we use a classical coordinatization of Desarguesian projective planes developed by Hilbert.
\item[(d)] We use the action of $\XX$ on $\Pp$ to construct a matrix representation
\[
    \XX \longrightarrow \so_3(\KK) .
\]
where $\so_3(\KK)$ is realized as a group of $3\times 3$ matrices over the black box field $\KK$.

\item[(e)] Coordinatizing $\so_3(\KK)$ in a similar way, we construct  an isomorphism
\[
     \so_3(\KK) \longrightarrow \XX.
\]
\item[(f)]   The map
\[
\so_3(\bF_q) \longrightarrow \so_3(\KK)
\]
is constructed from the isomorphism
\[
\bF_q \longrightarrow \KK
\]
from the explicitly given finite field $\bF$ onto a black box field $\KK$.  We use a result by Maurer and Raub  \cite{maurer07.427} formulated in our paper as Theorem~\ref{th:bbfiekds} to construct this isomorphism. This is a polynomial time algorithm which runs in time $\log |\bF_q|$. (All known algorithms for the inverse isomorphism
\[
\bF_q \longleftarrow \KK
\]
are reduced to solving discrete logarithm problem in $\bF_p$, the prime subfield of $\bF_q$, so finding polynomial time algorithms which construct two-way isomorphisms represent a major open problem in algebraic cryptography).
\item[(g)] Section \ref{sec:complexity} contains analysis of  complexities of all algorithms used in the proof.
\ei

\subsection{Monte-Carlo  and Las Vegas algorithms}\label{sec:Monte-Carlo}

Recall that a \textit{Monte-Carlo algorithm} is a randomized algorithm which gives a correct output  with probability strictly bigger than $1/2$.
A special case of Monte--Carlo algorithms is a \textit{Las Vegas algorithm} which either outputs a correct answer or reports failure. A detailed comparison of Monte--Carlo and Las Vegas algorithms, both from practical and theoretical point, can be found in \cite{babai97.1}.

By the nature of our axioms, many algorithms for black box groups (in the sense of Axioms BB1--BB4) are Monte-Carlo. In case of decision problems, where the output is ``yes'' or ``no'', the error of a Monte--Carlo algorithm can be made arbitrarily close to 0 by repeatedly re-running it. So answers to questions of the kind ``Is this black box group isomorphic to the given group?''  can be made as precise as we wish. In this paper, algorithms are Las Vegas.

\subsection{Terminology and notation}

In what follows we  make extensive use of the language of projective geometry, see, for example Coxeter \cite{coxeter2003projective} and Hartshorne \cite{hartshorne67}. Group theoretic terminology mostly follows \cite{gorenstein1994}.

\subsection{Organization of the paper}
In Section \ref{sec:blackboxgroups}, we discuss the axioms of black box groups.   Section \ref{sec:bbfields} contains a brief discussion of black box fields. In Section \ref{sec:moraproto}, we introduce morphisms and protomorphisms of black box groups and the crucially important procedure that we call  ``reification of an involution". We also explain how our arguments work in even characteristic producing a unipotent element in $\psl_2(2^n)$. In Section \ref{sec:application-reification}, we present applications of reification of involutions, in particular, we construct a black box group encrypting $\so_3(\bF)$ from a black box group encrypting $\psl_2(\bF)$. In Section \ref{sec:geometryofinvolutions}, we discuss the geometry of involutions in $\so_3(\bF)$, and in Section \ref{sec:projective-plane} we construct a black box projective plane and present algorithms for additional operations and relations coming from the underlying Lie algebra.   In Section \ref{sec:sym4}, we construct a black box subgroup encrypting $\Sym_4$ in a black box group encrypting $\so_3(\bF)$ (it provides us with an orthogonal basis in the projective plane with a polarity) and in Section \ref{sec:coordinatisation}, we apply Hilbert's coordinatization to the black box projective plane and construct a black box field. In Section \ref{sec:serendipity}, we prove Theorem \ref{cor:unipotent} and in Section \ref{sec:coordinatisationaction}, we prove Theorem \ref{SL2-SO3}. In Section \ref{sec:complexity}, we give the complexities of the procedures presented in this paper.

\section{Black box groups}\label{sec:blackboxgroups}

\subsection{Axioms for black box groups}\label{sec:BB1--BB3}

What follows are slightly modified Babai-Szmer\'{e}di axioms.

The concept of a black box can be applied to rings, fields, and, as we can see in this  paper, even to projective planes. So, we formulate our axioms for groups but use the wording which makes them applicable to other algebraic structures. This explains why we are using the length $l(\XX)$ of strings produced by a black box $\XX$ as a proxy for the complexity of a black box $\XX$: it is applicable to a variety of structures.

In our algorithms, we have to work with several black box groups at once and build new black boxes (sometimes for the same abstract group, sometimes for different groups) from existing ones. For that reason we specify the functionality of \emph{a family} $\mathcal{X}$ of \emph{black boxes} $\XX$ by the following axioms.

\begin{itemize}
\item[\textbf{BB1}] On request, each $\XX$ produces a binary string of fixed length $l(\XX)$ (which depends on $\XX$) encrypting a random (almost) uniformly distributed element from some fixed group $X$; this is done in probabilistic time polynomial in $l(\XX)$.
\item[\textbf{BB2}] Each $\XX$ computes, in probabilistic time polynomial in $l(\XX)$, a string encrypting the product of two strings or an inverse of a string (that is, a string encrypting the inverse of an element given by a string).
\item[\textbf{BB3}] Each $\XX$ decides, in probabilistic time polynomial in $l(\XX)$, whether two strings encrypt the same element in its group $X$ -- therefore identification of strings agrees with the canonical projection
\begin{diagram} \XX & \rDotsto^{\pi_\XX} & X.
 \end{diagram}
\end{itemize}

If Axioms BB1--BB3 hold for a particular black box group $\XX$, we say that $\XX$ is a \emph{black box over $X$}, or that a black box $\XX$ \emph{encrypts} the group $X$. Notice that we are not making any assumptions of practical computability or the time complexity of the projection $\pi_\XX$. We will discuss our set up and terminology further in Section \ref{sec:blackboxsubgroups}.

In respect of Axiom BB1, we note here that a black box group $\XX$ encrypting a finite group $X$ may not necessarily be given by generators: see, for example, discussion of black boxes for centralizers of involutions \cite{borovik02.7}, or discussion of links between black box algebra and homomorphic encryption \cite{BY2017C}. When a black box group $\XX$ is given by some generators, that is, some strings $\bx{x}_1,\bx{x}_2,\ldots, \bx{x}_m$ in $\XX$ such that $X=\langle \pi_\XX(\bx{x}_1), \pi_\XX(\bx{x}_2), \dots, \pi_\XX(\bx{x}_m)\rangle$, then producing random elements from $\XX$ can be done by using either the algorithm presented in \cite[Theorem 1.1]{babai91.164} or the algorithm called ``the product replacement algorithm'' \cite{celler95.4931}. It turned out that the product replacement algorithm is much more practical and we refer reader to \cite{babai04.215, lubotzky01.347, pak.01.476, pak01.301} for its detailed analysis.

In this paper, we systematically build new black box groups from old ones, and use randomized algorithms for their constructions. In this situation, operations in these new black boxes are performed by randomized algorithms -- this explains the randomization introduced in Axioms BB2 and BB3.

A typical example of a black box group is provided by a group $X$ generated in a big matrix group $\gl_n(r^k)$ by several matrices $x_1,\dots, x_l$. We can, of course, multiply, invert, compare matrices. Therefore computer routines for these operations together with the sampling of the product replacement algorithm run on a  tuple of generators $(x_1,\dots, x_l)$ can be viewed as a black box $\XX$ encrypting the group $X$.

\subsection{Global exponent and Axiom BB4}\label{sec:BB4}

Notice that, even in routine examples, the number of elements of a matrix group $X$ could be astronomical. This makes many natural questions about the black box $\XX$ over $X$ -- for example, finding the order of $X$ or the isomorphism type of $X$ when $X$ is given as a simple group of Lie type  -- inaccessible for all known deterministic methods. Even when $X$ is cyclic, existing approaches to finding its order are conditional and involve either the discrete logarithm problem or  prime factorization of large  integers.

Nevertheless black box problems for finite groups frequently have a feature which makes them more accessible:

\begin{itemize}
\item[\textbf{BB4}] We are given a \emph{global exponent} of $\XX$, that is, a natural number $E:=E(\XX)$ such that
\begin{itemize}
\item[$\bullet$] $\bx{x}^E = 1$ for all strings $\bx{x}$ produced by  $\XX$; and
\item[$\bullet$] $\log E$ is polynomially bounded in terms of $l(\XX)$.
\end{itemize}
\end{itemize}

For example, if $\XX$ is a black box group arising from a subgroup in the known ambient group $G$, the exponent of $G$ can be taken for a global exponent of $\XX$.

If we know the factorization of $E$ into prime factors then we can find the order of any element $\bx{x}$  produced by $\XX$ as the minimal divisor $e$ of $E$ such that $\bx{x}^e=\bx{1}$. However, we wish to work with linear groups over fields of large characteristic where factorization of $E$ is becoming unfeasible. Our approach allows us to avoid determination of orders of random elements from $\XX$ and consequently avoid making any assumptions about the prime factorization of the global exponent.

\subsection{Axiom BB5} \label{sec:Euclid-Lobachevsky}

It is important to observe that our proof of Theorem~\ref{SL2-SO3} uses the global exponent $E$ and Axiom BB4 only for computing square roots of semisimple elements in $\YY$ and $\XX$ (this is done by Tonelli-Shanks algorithm, Lemma~\ref{lm:Tonelli-Shanks}). Therefore Axiom BB4  can be replaced by its corollary, Axiom BB5 -- see Theorem \ref{th:square-roots}.

\begin{itemize}
\item[\textbf{BB5}] We are given a partial $1$- or $2$-valued function $\rho$ of two variables on a subset $\SS\subset \XX$ that computes, in probabilistic time polynomial in $l(\XX)$, square roots in cyclic subgroups of $\XX$ in the following sense:
    \bq
    if $\bx{x}\in \SS$ and $\bx{y}\in \langle \bx{x}\rangle$ has square roots in $\langle \bx{x}\rangle$ then $\rho(\bx{x},\bx{y})$ is the set of these roots.
    \eq
\end{itemize}

In particular,
\bi
\item
if $|\bx{x}|$ is even, $\rho(\bx{x},1)$ is the subgroup of order $2$ in $\langle \bx{x} \rangle$;
\item if $|\bx{x}|$ is even then, consecutively applying $\rho(\bx{x}, \cdot)$ to $2$-elements in $\langle \bx{x} \rangle$, we can find $2$-elements in $\langle \bx{x} \rangle$ of every order present;
\item if $|\bx{x}|$ is odd, and $\bx{x}\in \langle \bx{x}\rangle$ then $\rho(\bx{x},\bx{x})$ is the unique square root of $\bx{x}$ in $\langle \bx{x} \rangle$.
\ei

We emphasize that Axiom BB5 provides everything needed for construction of centralizers of involutions by the maps $\zeta_0$ and $\zeta_1$, Section~\ref{sec:centralizerofprotoinvolution}.

Axiom BB5 follows from BB4 by  Lemma~\ref{lm:Tonelli-Shanks}, applied to the cyclic group $\langle x \rangle$.

\section{Black box fields} \label{sec:bbfields}

We define black box fields using, by analogy with black box groups,  Axioms BB1--BB3, with a few obvious changes in the wording,  and with Axiom BB2 covering the addition, multiplication, and inversion in the field. The reader may wish to compare our exposition with \cite{boneh96.283}.  We remind that, in this paper, we do not necessarily know the characteristic of the field. Therefore we slightly generalize the definition of a black box field given in \cite{boneh96.283, maurer07.427} by removing the assumption that the characteristic of the field is known. We refer the reader to \cite{boneh96.283, maurer07.427} for more details of black box fields of known characteristic.

We shall be using some results about the isomorphism problem for black box fields of known characteristic $p$ \cite{maurer07.427}, that is, the problem of constructing an isomorphism and its inverse between a black box field $\KK$ and an explicitly given finite field $\bF_{p^n}$. The explicit data for a finite field of cardinality $p^n$ is defined to be a system of {\emph{structure constants} over the prime field, that is,} $n^3$ elements $(c_{ijk})_{i,j,k=1}^n$ of the prime field {$\mathbb{F}_p = \mathbb{Z}/p\mathbb{Z}$ (represented as integers in $[0,p-1]$)} so that $\mathbb{F}_{p^n}$ becomes a field with ordinary addition and multiplication by elements of $\mathbb{F}_p$, and multiplication  determined by
\[
s_i s_j =\sum_{k=1}^n c_{ijk}s_k,
\]
where $s_1, s_2, \dots, s_n$ denotes a basis of $\mathbb{F}_{p^n}$ over $\mathbb{F}_p$. The concept of an explicitly given field of order $p^n$ is robust; indeed, Lenstra Jr.\  has shown in \cite[Theorem 1.2]{lenstra91.329} that for any two fields $A$ and $B$ of order $p^n$ given by two sets of structure constants $(a_{ijk})_{i,j,k=1}^n$ and $(b_{ijk})_{i,j,k=1}^n$ an isomorphism $A \longrightarrow B$ can be constructed in time polynomial in $n\log p$.

By an \textit{efficient isomorphism} between a black box field and an explicitly given finite field $\bF_{p^n}$, we mean an algorithm constructing such an isomorphism in time polynomial in $n$ and $\log p$.

One of the key results on black box fields belongs to Maurer and Raub  \cite{maurer07.427}; its statement and proof can be reformulated to yield the following result.

\begin{theorem} \label{th:bbfiekds}
Let\/ $\KK$  be a black box field of known characteristic $p$ encrypting an explicitly given finite field\/ $\mathbb{F}_{p^n}$ and $\KK_0$ the  prime subfield of\/ $\KK$. Then the isomorphism problem between $\KK$ and\/ $\mathbb{F}_{p^n}$ can be efficiently reduced to the isomorphism problem between $\KK_0$ and\/ $\mathbb{F}_p$. In particular,
\bi
\item  an efficient isomorphism $\KK_0 \longrightarrow \mathbb{F}_p$ can be extended in time polynomial in the input length\/ $l(\KK)$ to an efficient isomorphism
$\KK \longrightarrow \mathbb{F}_{p^n};$
\item there exists an isomorphism  $\mathbb{F}_{p^n} \longrightarrow \KK$ computable in polynomial in $l(\KK)$  time.
    \ei
\end{theorem}
The existence of an efficient isomorphism $\KK_0 \longrightarrow \bF_p $ would follow from solution of the discrete logarithm problem in $\KK_0$. In particular, this means that, for small primes $p$, every black box field of order $p^n$ is effectively isomorphic to $\bF_{p^n}$.

\section{Morphisms and protomorphisms} \label{sec:moraproto}

This section contains crucial tools for our algorithms. They are based on a simple observation that a map \begin{diagram}
X & \rDotsto^{\phi} & Y
\end{diagram} from a group to a group is a homomorphism of groups if and only if its graph
 \[
F = \{(x,\phi(x)): x\in X\}
\]
is a subgroup of $X \times Y$. Essentially we treat homomorphisms of black box groups as black box groups on their own, see Section \ref{sec:morphisms:bbg}. With this principle, almost everything in this section is self-evident.

First, we introduce some terminology.

\subsection{Morphisms} \label{sec:morphisms}

Given two  black boxes $\XX$ and $\YY$ encrypting  finite groups $X$ and  $Y$, respectively, we say that a map $\bxg{\phi}$ which assigns strings produced by $\XX$  to strings produced by $\YY$  is a \emph{morphism} of black box groups, if
\bi
\item the map $\bxg{\phi}$ is computable in probabilistic time polynomial in $l(\XX)$ and $l(\YY)$, and
\item there is a homomorphism $\phi:X \to Y$ such that the following diagram  is commutative:
\begin{diagram}
\XX &\rTo^{\bxg{\phi}} &\YY\\
\dDotsto_{\pi_{\XX}} & &\dDotsto_{\pi_{\YY}}\\
X &\rTo^{\phi} & Y
\end{diagram}
where $\pi_\XX$ and $\pi_\YY$ are the canonical projections of $\XX$ and $\YY$ onto $X$ and $Y$, respectively.
\ei

We shall say in this situation that a morphism $\bxg{\phi}$ \emph{encrypts} the homomorphism $\phi$. For example, morphisms arise naturally when a black box group $\XX$ is given by a generating set and we replace a generating set for the black box group $\XX$ by a more convenient one and run the product replacement algorithm for the new generating set; in fact, we replace a black box for $\XX$  and deal with a morphism $\YY \longrightarrow \XX$ from a new black box $\YY$ into $\XX$.

We apply to morphisms and abstract homomorphisms the same notational convention as to the strings and elements, using the same letters in sans serifed
or plain version, respectively.

Since different strings produced by $\YY$ may represent the same element in $Y$, replacing strings $\bxg{\phi}(\bx{x})$ by equivalent strings produces a new morphism $\bxg{\phi'}$ which also encrypts $\phi$ and, for all the practical purpose is the same as $\bxg{\phi}$. 

Slightly abusing terminology, we say that a morphism $\bxg{\phi}$ is an injection, or a surjection, etc., if  $\phi$ has these properties. In accordance with standard conventions, hooked arrows $\hookrightarrow$ stand for injections; dotted arrows are reserved for homomorphisms, including natural projections
\begin{diagram}
\XX & \rDotsto^{\pi_\XX} & X;
\end{diagram}
which are not necessarily morphisms, since, by the very nature of black box problems, we are not given an efficient procedure for constructing the projection of a black box onto the group it encrypts.

\subsection{Black box subgroups and further remarks on terminology} \label{sec:blackboxsubgroups}

If we have an injection
\begin{diagram}
\YY &\rInto^{\bxg{\phi}} &\XX\\
\dDotsto_{\pi_{\YY}} & &\dDotsto_{\pi_{\XX}}\\
Y &\rInto^{\phi} & X
\end{diagram}
we say that $\YY$ is a \emph{black box subgroup} of $\XX$ encrypting $Y\leq X$. We emphasize that a black box subgroup is a procedure and has to be treated as such especially when one writes a computer code for black box group algorithm. Notice that different black box subgroups may encrypt the same subgroup. Indeed, an element in $Y$ can be encrypted by several different strings produced  from $\XX$; it is important to take into consideration a possibility that not all of these strings are produced by $\YY$.

If elements $\pi_\YY(\bx{y}_1), \dots, \pi_\YY(\bx{y}_k)$ generate $Y$, we call strings $\bx{y}_1, \dots, \bx{y}_k$ \emph{generators} of $\YY$.

Black box subgroups will be constructed in this paper in one of the following ways:

\bi
\item We pick some strings $\bx{y}_1,\dots,\bx{y}_m$ produced by the black box group $\XX$ and we treat them as generators of a black box subgroup $\YY$. We use the product replacement algorithm \cite{celler95.4931} for random sampling.
\item Given black box subgroups $\YY_1,\dots,\YY_k$ in $\XX$, we generate a subgroup $$\YY = \langle \YY_1,\dots,\YY_k\rangle$$ by taking generating sets in $\YY_i$ and combining them into a generating set in $\YY$.
\item If $\YY$ is the centralizer in $\XX$ of an involution or a proto-involution in the sense of Section~\ref{sec:proto-involution} then we apply the procedure described in Section \ref{sec:centralizerofprotoinvolution} to ``populate'' $\YY$ and eventually find a generating set for $\YY$.
\ei

\paragraph{\bf Terminology and conventions.}  Abusing terminology and notation, we write $\bx{x}\in \XX$ for a string $\bx{x}$ produced by the black box group $\XX$ and we say that $\bx{x}$ is an \emph{element} of $\XX$. In the rest of the paper, a subgroup of a black box group is meant to refer to a black box subgroup. The \emph{order} $o(\bx{x})$ ($= |\bx{x}|$) of $\bx{x}$ is $o(x)=|x|$. We shall refer to a string $\bx{x} \in \XX$ as an \emph{involution}, or a semisimple element, or a unipotent element, etc.,  if $x \in X$ is an element with these properties. If $\bx{x} \in \XX$ is an involution, then $\CC_\XX(\bx{x})$ denotes a black box subgroup encrypting $C_X(x)$, see Section \ref{sec:centralizerofprotoinvolution} for a construction of $\CC_\XX(\bx{x})$. Moreover, we refer to black box subgroups in $\XX$ as to tori, unipotent groups, etc., if they encrypt subgroups in $X$ with these properties.

\subsection{Direct and semidirect products of black box groups}\label{subsec:direct}

Assume that $\XX$ encrypts $X$ and $\YY$ encrypts $Y$. Then the black box  $\XX\times \YY$ produces pairs of strings $(\bx{x},\bx{y})$ by sampling $\XX$ and $\YY$ independently, with operations carried out componentwise in $\XX$ and $\YY$; of course, $\XX\times \YY$ encrypts $X\times Y$ and $l(\XX \times \YY)=l(\XX)+l(\YY)$.

More generally, given black box groups $\XX_1,\dots, \XX_n$, we can define their direct product
\[
\XX = \XX_1\times\cdots\times\XX_n
\]
in an expected way, consecutively sampling strings $\bx{x}_i \in \XX_i$ to form a random $n$-tuple $(\bx{x}_1,\dots,\bx{x}_n)$ and carrying out group operations on these $n$-tuples componentwise.

Later in the paper, we will use semidirect products of black box groups. They arise in a situation when we have two black box group $\XX$ and $\YY$ and a polynomial time in $\l(\XX)$ and $l(\YY)$ procedure for the action of $\YY$ on $\XX$ by automorphisms,
\[
\XX \times \YY \longrightarrow  \XX, \qquad
(\bx{x},\bx{y})  \mapsto  \bx{x}^\bx{y};
\]
then $\XX \rtimes \YY$ samples independent pairs $(\bx{x},\bx{y})$ of strings from $\XX$ and $\YY$ with  multiplication performed and inversion  by the rules
\begin{equation}\label{eq:semi1}
(\bx{x}_1,\bx{y}_1) \circ (\bx{x}_2,\bx{y}_2) := (\bx{x}_1\bx{x}_2^{\bx{y}_1} \!,\, \bx{y}_1\bx{y}_2)
\; \mbox{ and } \; (\bx{x},\bx{y})^{-1} := ((\bx{x}^{-1})^{\bx{y}^{-1}}, \bx{y}^{-1}).
\end{equation}

\subsection{Morphisms as black box groups} \label{sec:morphisms:bbg}

Given a morphism
\begin{diagram}
\XX &\rTo^{\bxg{\phi}} &\YY\\
\dDotsto_{\pi_{\XX}} & &\dDotsto_{\pi_{\YY}}\\
X &\rTo^{\phi} & Y
\end{diagram}
of black box groups, we can associate with it a black box subgroup  $\ZZ\hookrightarrow \XX\times \YY$ which encrypts the graph  $F=\{(x, \phi(x) : x \in X)\}$ of $\phi $. The black box group $\ZZ$ produces strings $\{(\bx{x}, \bxg{\phi}(\bx{x}))\}$ with $\bx{x}$ is sampled by the black box $\XX$
and the natural projection is defined as
\bea
\pi_\ZZ: \ZZ & \longrightarrow & F\\
(\bx{x},\bxg{\phi}(\bx{x})) & \mapsto & (\pi_\XX(\bx{x}), \phi(\pi_\XX(\bx{x})).
\eea

In practice, this means that we find strings $\bx{x}_1,\dots,\bx{x}_k$ generating $\XX$ with known images $\bx{y}_1=\bxg{\phi}(\bx{x}_1),\dots, \bx{y}_k=\bxg{\phi}(\bx{x}_k)$ in $\YY$ and then use the product replacement algorithm for the black box  subgroup \[\ZZ^* = \langle (\bx{x}_1,\bx{y}_1),\dots,(\bx{x}_k,\bx{y}_k)\rangle \hookrightarrow \XX \times \YY\] encrypting a subgraph $\{\, (\bx{x},\bxg{\phi}(\bx{x}))\,\}$ of the homomorphism $\bxg{\phi}$. Random sampling of the black box $\ZZ^*$ returns strings $\bx{x}\in \XX$ with their images $\bxg{\phi}(\bx{x})\in \YY$ already attached.

\subsection{Protomorphisms}

Let $\XX$ and $\YY$  be two black box groups encrypting $X$ and $Y$, respectively, and $\pi$ the canonical projection of $\XX\times \YY$ onto $X\times Y$. A \emph{protomorphism} $\ZZ$ between black box groups  $\XX$ and $\YY$ is a black box subgroup $\ZZ \hookrightarrow \XX \times \YY$ such that $\pi_\ZZ(\ZZ)$ is  the graph of a homomorphism from $X$ to $Y$ or from $Y$ to $X$ -- the direction of homomorphism is not set here.
We say that $\ZZ$ \emph{encrypts} this homomorphism.

Given a string $\bx{x}$ in a black box group $\XX$ encrypting a group $X$, it is frequently useful to associate with $\bx{x}$ a black box for the graph of a specific automorphism of $X$, namely, the conjugation by $\pi_\XX(\bx{x})$. It can be  viewed as a black box subgroup $\CC_\bx{x} \hookrightarrow \XX\times \XX$, which produces strings $(\bx{y},\bx{y}^\bx{x})$ for random strings $\bx{y} \in \XX$, with group operations and equality relation defined in the obvious way.

Treating a homomorphism $\XX \longrightarrow \YY$  of black box groups $\XX$ and $\YY$  as a black box subgroup in their direct product $\XX \times \YY$ allows us to construct previously inaccessible objects -- see, for example, ``reification of involutions'',  Section~\ref{sec:reification}.

\subsection{Amalgamation of local proto-automorphisms} \label{sec:proto-involution}

Let $\XX$  be a black box group encrypting a group $X$. Expanding the terminology from the previous section, a proto-automorphism $\FF$ on  $\XX$  is a black box subgroup $\FF \hookrightarrow \XX\times \XX$ for the graph of an automorphism of $\XX$.

Assume that  black box subgroups $\YY_1,\dots,\YY_k$ of $\XX$ are encrypting, respectively, subgroups $Y_1,\dots,Y_k$ of $X$, and assume that $\langle Y_1,\dots, Y_k \rangle = X$. Assume that $\phi_1,\dots, \phi_k$ are  automorphisms of subgroups $Y_1,\dots,Y_k$, respectively, and $\FF_i$ are proto-automorphisms on $\YY_i$ encrypting $\phi_i$, $i=1,\dots,k$. We say that the system of proto-automorphisms $\FF_1,\dots,\FF_k$ is \emph{consistent} if there exists a unique automorphism $\phi$ of $X$ such that $\phi_i = \phi\mid_{Y_i}$ for all $i=1,\dots,k$.

\begin{theorem}[Amalgamation of local proto-automorphisms] \label{tH:algamation-of-local-involutions}
If\/ $\FF_1,\dots,\FF_k$ is a consistent system of proto-automorphisms on black box subgroups in $\XX$, then $$\FF =\langle \FF_1,\dots,\FF_k\rangle$$ is a proto-automorphism on $\XX$.
\end{theorem}

\begin{proof}
The proof is self-evident.
\end{proof}

We call $\FF$ the \emph{amalgam} of proto-automorphisms $\FF_1,\dots,\FF_k$.

\begin{theorem}[Augmentation of a black box group by a proto-involution] \label{th:augmentation-by-involution} If\/ $$\AA \hookrightarrow \XX\times\XX$$ is a proto-automorphism on $\XX$ encrypting an involutive automorphism $\alpha$ on $X$, we can construct an involutive automorphism $\bxg{\alpha}$ of $\AA$ by setting
\[
\bxg{\alpha}: (\bx{x},\bx{x}') \mapsto (\bx{x}',\bx{x}) \mbox{ for } (\bx{x},\bx{x}') \in \AA.
\]
Then the semidirect product $\AA \rtimes \{1, \bxg{\alpha}\}$ is a black box encrypting $X\rtimes \langle \alpha \rangle$, with $\AA$ canonically projecting onto $X$ and $\bxg{\alpha}$ projecting to $\alpha$.
\end{theorem}
\begin{proof}
The proof is self-evident.
\end{proof}

Theorems~\ref{tH:algamation-of-local-involutions} and \ref{th:augmentation-by-involution} provide the conceptional frame for a construction of a
black group encrypting $\so_3(\bF) \simeq \pgl_2(\bF)$ from a black box group encrypting $\psl_2(\bF)$, see Theorem~\ref{psl2-pgl2}.

\subsection{Centralizer of a proto-involution}\label{sec:centralizerofprotoinvolution}

Let $\AA \hookrightarrow \XX\times\XX$ be a proto-involution on $\XX$ defined in Section \ref{sec:proto-involution}. Assume that $\AA$ encrypts some involutive automorphism $\alpha \in {\rm Aut}(X)$ of $X$. Later in the paper, we work in the situation when $\alpha$ is an external automorphism and $\bxg{\alpha}$ is not present as a string in the black box group. In this case, we do not have any access to $\bxg{\alpha}$ but we can make use of its graph.
Given $\AA$, we shall construct a black box group, which we denote it as $\CC_\XX[\AA]$, encrypting $C_X(\alpha)$.  Let $A=\{(x,x^\alpha)\in X\times X \mid x\in X\}$ be the graph of $\alpha$.  Obviously, by writing formally
\[
C_X[A]=\{ x\in X \mid \mbox{ there exists } y \in X \mbox{ such that } (x,y) \in A \mbox{ and }  x=y \},
\]
we get $C_X(\alpha)=C_X[A]$. We will construct a black box subgroup $\CC_\XX[\AA] \hookrightarrow \XX$ which satisfies a similar condition
\[
\CC_\XX[\AA]=\{\bx{x}\in \XX \mid \mbox{ there exists } \bx{y} \in \XX \mbox{ such that } (\bx{x},\bx{y}) \in \AA \mbox{ and }  \pi_\XX(\bx{x})=\pi_\XX(\bx{y}) \},
\]
and encrypts $C_X[A]$.

It follows from the arguments in \cite{borovik02.7, bray00.241} that we have the map $\zeta = \zeta_0\sqcup \zeta_1$:
\begin{eqnarray*}
\zeta:  X & \longrightarrow &  C_X(\alpha)\\
 x & \mapsto & \left\{ \begin{array}{ll}
\zeta_0(x) = {\rm i}(x^{\alpha}x^{-1})  &  \hbox{ if } o(x^{\alpha}x^{-1}) \hbox{ is even}\\
\zeta_1(x) = \sqrt{x^\alpha x^{-1}}\cdot x & \hbox{ if } o(x^\alpha x^{-1}) \hbox{ is odd,}
\end{array}\right.
\end{eqnarray*}
where ${\rm i}(x)$ is the unique involution and $\sqrt{x}$ is the unique square root of $x$ in the cyclic group $\langle x \rangle$ of odd order.

Notice that the map above can be written in the following way using the graph $A$ of $\alpha$ instead of $\alpha$.
\begin{eqnarray*}
\zeta:  A & \longrightarrow &  C_X[A]=C_X(\alpha)\\
 (x,y) & \mapsto & \left\{ \begin{array}{ll}
\zeta_0((x,y)) = {\rm i}(yx^{-1})  &  \hbox{ if } o(yx^{-1}) \hbox{ is even}\\
\zeta_1((x,y)) = \sqrt{yx^{-1}}\cdot x & \hbox{ if } o(yx^{-1}) \hbox{ is odd.}
\end{array}\right.
\end{eqnarray*}

Assume that $\XX$ satisfies Axiom BB4 and has a global exponent $E=2^km$ with $m$ odd.

If $\bx{x}\in\XX$ is a string encrypting an element of even order then the last non-identity string in the sequence
$$1 \neq \bx{x}^{m}, \, (\bx{x}^m)^2, \, (\bx{x}^m)^{2^2}, \dots, \left(\bx{x}^m\right)^{2^{k-1}}, \left(\bx{x}^{m}\right)^{2^k}=1$$
is an involution and denoted by ${\rm i}(\bx{x})$.

If  $\bx{x}\in\XX$ encrypts an element of odd order then $\bx{y} := \bx{x}^{(m+1)/2}$ obviously satisfies $\bx{y}^2=\bx{x}$ and is a square root of $\bx{x}$ in $\langle \bx{x}\rangle$; we denote $\bx{y} := \sqrt{\bx{x}}$.

Hence, we have the analogous map for the black box groups.

\begin{eqnarray}
\bxg{\zeta}:  \AA & \longrightarrow &  \CC_{\XX}[\AA]\nonumber\\
 (\bx{x},\bx{y}) & \mapsto & \left\{ \begin{array}{ll}
\bxg{\zeta}_0((\bx{x},\bx{y})) = {\rm i}(\bx{y}\bx{x}^{-1})  &  \hbox{ if } o(\bx{y}\bx{x}^{-1}) \hbox{ is even}\\
\bxg{\zeta}_1((\bx{x},\bx{y})) = \sqrt{\bx{y}\bx{x}^{-1}}\cdot \bx{x} & \hbox{ if } o(\bx{y}\bx{x}^{-1}) \hbox{ is odd.}
\end{array}\right.
\label{eqn:zeta}
\end{eqnarray}

If $X$ is a simple group of Lie type then, as shown in \cite{parker10.885}, $\zeta_1(x)$ is defined for random $x\in X$ with probability  $O(1/n)$ where $n$ is the Lie rank of $X$. Furthermore, the same calculation as in \cite[Section 6]{borovik02.7} proves that elements $\pi_\XX(\bxg{\zeta}_1(\bx{x},\bx{y}))$, for $(\bx{x},\bx{y})\in \AA$, are uniformly distributed over $C_X(\alpha)$. Therefore $\bxg{\zeta}_1$ provides an efficient black box for $\CC_{\XX}[\AA]$.

If $\AA$ is a graph of an inner automorphism corresponding to an element $\bx{a} \in \XX$, we denote $\CC_{\XX}[\AA]$ by $\CC_\XX(\bx{a})$.

For the black box groups $\XX$ encrypting $\psl_2(\bF)$ or $\pgl_2(\bF)$, where $\bF$ is a finite field of odd characteristic, we can construct a generating set for $\CC_\XX(\bx{a})$ in the following way, if needed. We know that $\CC_\XX(\bx{a}) = \TT\rtimes \langle \bx{w}\rangle$, where $\TT$ is the maximal torus containing the involution $\bx{a}$ and $\bx{w}$ inverts $\TT$. Since the map $\bxg{\zeta}_1$ produces uniformly distributed elements in $\XX$, by \cite[I.8]{mitrinovic1996}, a set of size $O(\log \log |\bF|)$ consisting of random elements in $\CC$ contains a generator for $\TT$ with probability $>1/2$. Moreover, since the half of the elements in $\CC$ are involutions inverting $\TT$, we can construct a generating set of size $O(\log \log |\bF|)$ for $\CC_\XX(\bx{a})$.

The map $\bxg{\zeta}_0$ is useful when we are interested mostly in involutions in $\CC_{\XX}[\AA]$, as it happens, for example, in reification of involutions, see Section~\ref{sec:reification}.

\subsection{Reification of an involution} \label{sec:reification}

We approach the most fascinating part of the story: identification of an involution in $\XX$ from its description.
We shall call this procedure the \emph{reification of an involution}; it is intensively used in the present paper and in \cite{BY2017B}.

Following the notation from the previous subsection, assume that $\AA \hookrightarrow \XX\times\XX$ is a proto-involution on $\XX$ encrypting a specific inner automorphism of $X$,  conjugation by an involution $a\in X$. We want to find a string $\bx{a}$  in $\XX$ that encrypts $a$. Let $A < X \times X$ be the graph of conjugation by $a$. Obviously, $a \in Y = C_X[A]$, and, moreover, $a \in Z(Y)$. Denote by $\Omega(Y)$ the subgroup generated by all involutions in $Z(Y)$, then $a\in\Omega(Y)$. Even more so, a probabilistic algorithm described in Section~\ref{sec:centralizerofprotoinvolution} gives us a black box group $\YY = \CC_\XX[\AA]$ encrypting $Y$. Now we apply to the subgroup (unknown to us) $\Omega(Y)$ the algorithm proving the following lemma.

\begin{lemma}
\label{lm:locating-central-involution}
Assume that a black box group\/ $\YY$ encrypts a subgroup\/ $Y$ in a simple group $X$ of Lie type of odd characteristic and of Lie rank $n$. Assume also that we know a global exponent $E$ for $\YY$. Then there is a Las Vegas algorithm  which constructs a black box group $\WW \hookrightarrow \YY$ which encrypts an elementary abelian $2$-group $W \leq Y$ which contains $\Omega(Y)$. The algorithm works in probabilistic time polynomial in $l(\YY)$, $\log E$ {\rm (}the global exponent of $\YY${\rm )}, and $n$.
 \end{lemma}

\begin{proof}
Let $\bx{i}_1$ be an involution produced from a random element in $\YY$ by repeated square and multiply method. Then $\YY_1=\CC_\YY(\bx{i}_1)$ is a black box subgroup encrypting a subgroup $Y_1\leq Y$ containing $\Omega(Y)$. Now, let $\bx{i}_2$ be a string encrypting an involution different from $\bx{i}_1$  produced by $\YY_1$, then similarly $\YY_2=\CC_{\YY_1}(\bx{i}_2)$ is a black box subgroup encrypting a subgroup $Y_2\leq Y_1$ containing $\Omega(Y)$. Continuing in this way, we descend to a black box subgroup $\ZZ \hookrightarrow \YY$ encrypting an abelian group $Z$ containing $\Omega(Y)$. Now, by using the black box group $\ZZ$, we construct the involutions in $\ZZ$ as described in Section \ref{sec:centralizerofprotoinvolution}.  Finally, we use these involutions and bounds from \cite[pp. 192--193]{Pomerance2002} to construct a generating set in a black box subgroup $\WW\hookrightarrow \YY$ which encrypts an elementary abelian subgroup containing $\Omega(Y)$.

If $X$ is a simple group of Lie type of odd characteristic, then the length of chains of centralizers of involutions is bounded by a  polynomial in its Lie rank, giving a crude upper bound of $\log |X|$. Since elements of even order (hence involutions) in $X$ are abundant by \cite{isaacs95.139} and the number of involutions in $Z(Y)$ is bounded by a polynomial in the Lie rank of $X$, the process quickly produces a desired black box subgroup $\WW \hookrightarrow \YY$.
\end{proof}

Since $\bx{a}$ is a string which encrypts an involution in an elementary abelian group in $Z(Y)$, it belongs to an elementary abelian group $\WW$ constructed in Lemma \ref{lm:locating-central-involution}; after that, $\bx{a}$ can be identified by testing every possibility in $\WW$. These crude estimates show that the reification procedure works in probabilistic time polynomial in $l(\XX)$, $\log E$ (where $E$ is the global exponent of $\XX$) and $|W|$.

In this paper, reification of involutions is applied to $\so_3(\bF)$ in odd characteristic. In this case, centralizers of involutions at the subsequent stages of the algorithm are either dihedral or abelian, and the black box subgroup $\WW$ in the statement of Lemma \ref{lm:locating-central-involution} has order at most 4. Hence the procedure, in this case, is pretty fast.

\subsection{Involutions in $\psl_2(2^n)$}\label{sec:even-case}

Let $\XX$ be a black box group encrypting $\psl_2(2^n)$ for some $n \geqslant 2$. A paper by Kantor and Kassabov \cite{kantor15.16} contains a construction of an involution in $\XX$, a  result analogous to the results in this paper. We shall now show how involutions in $\psl_2(2^n)$  can be constructed by our methods.

Take two non-commuting elements $\bx{x}$ and $\bx{y}$  of odd order. It is a well-known property of $\sl_2(2^n)$ that either $\bx{x}$ and $\bx{y}$ belong to the same Borel subgroup in $\XX$ -- but in that case $[\bx{x},\bx{y}]$ is an involution, and we are done, or
there is an involution $\bx{a} \in \XX$ which inverts both $\bx{x}$ and $\bx{y}$. We do not know $\bx{a}$, but it is obvious that the
proto-involution $\AA < \XX \times \XX$ corresponding to $\bx{a}$ contains
the tuple
$$((\bx{x}\bx{y})^\bx{a}, \bx{x}\bx{y}) = (\bx{x}^\bx{a}\bx{y}^\bx{a}, \bx{x}\bx{y})  = (\bx{x}^{-1}\bx{y}^{-1},\bx{x}\bx{y}).$$
Now it follows from Equation~(\ref{eqn:zeta}) that
\begin{eqnarray*}
\bxg{\zeta}_1((\bx{x}^{-1}\bx{y}^{-1},\bx{x}\bx{y})) = \sqrt{\bx{x}\bx{y}^2\bx{x}}\cdot \bx{x}^{-1}\bx{y}^{-1}.
\end{eqnarray*}
It is easy to see that the calculation produces an involution in $\CC_{\XX}(\AA)$ unless  $\bx{x}\bx{y}^2\bx{x}$ is already an involution.

\section{Applications of reifications of involutions} \label{sec:application-reification}

\subsection{Construction of $\so_3(\bF)$ from $\psl_2(\bF)$}\label{sec:constructionofpgl}

It will become clear later in this paper that black box  groups $\pgl_2(\bF)\cong \so_3(\bF)$ are easier to analyze than $\sl_2(\bF)$ or $\psl_2(\bF)$ because they contain more involutions. We extend a black box group encrypting $\psl_2(\bF)$ to a black box group encrypting $\so_3(\bF)$ using amalgamation of proto-automorphisms, Theorem~\ref{tH:algamation-of-local-involutions}, and augmentation of a black box group by a proto-involution, Theorem~\ref{th:augmentation-by-involution}.

If $\bx{x}\in \XX$ is a non-trivial semisimple element, then we will denote the maximal torus in $\XX$ containing $\bx{x}$ by $\TT_{\bx{x}}$.

\begin{theorem}\label{psl2-pgl2} Let $\YY$ be a black box group encrypting a group $Y = \psl_2(\bF)$, where $\bF$ is an unknown finite field of unknown odd characteristic and $|\bF| \geq 7$. Assume that we know a global exponent $E$ for $\YY$. Then there is an algorithm which constructs a proto-involution $\YY^\circ \hookrightarrow \YY \times \YY$ on $\YY$ encrypting a diagonal automorphism $d$ of $Y$. Moreover, if we take  the automorphism \[ \bx{d}: (\bx{x},\bx{x}') \mapsto (\bx{x}',\bx{x}) \mbox{ for } (\bx{x},\bx{x}') \in \YY^\circ \] of $\YY^\circ$, then the semidirect product $\XX =\YY^\circ \rtimes \langle \bx{d}\rangle$ is a black box group encrypting $\pgl_2(\bF)$.

The running time of the algorithm is polynomial in $\log E$.
\end{theorem}

\begin{proof}
We recall  that $Y$  has one conjugacy class of external involutive diagonal automorphisms \cite[Table 4.5.1]{gorenstein1998}. Let $d$ be its representative, then $C_Y(d) = S \rtimes \langle w \rangle$ where $S$ is a torus in $Y$ of order $(|\bF|-1)/2$ or $(|\bF|+1)/2$ depending on $\fmone$ or $\fpone$, respectively, and $w$ is an involution inverting $S$. Observe that the order of the torus $S$ is odd. Take an involution $t\in C_Y(d)$ inverting $S$ and assume that $t$ is contained in some maximal torus $T$. By the Frattini argument, $Y\cdot N_{Y\langle d\rangle}(T) = Y\langle d \rangle$ and we can assume without loss of generality that $d$ normalizes $T$.

Notice that $\langle T, S\rangle = Y$ and $d$ centralizes $S$ and inverts every element in $T$. Therefore we can apply to $\YY$ amalgamation of proto-automorphisms and  augmentation by a proto-involution,  Theorem~\ref{tH:algamation-of-local-involutions} and Theorem~\ref{th:augmentation-by-involution}.

Construction of tori $\TT$ and $\SS$ in $\YY$ with these properties goes as follows. We first construct an involution $\bx{u}\in \YY$ and its centralizer $\CC_\YY(\bx{u})=\TT_\bx{u} \rtimes \langle \bx{w}\rangle$. Then, we find a random element $\bx{y} \in \YY$ such that the element $\bx{z}:=\bx{u}\bx{u}^\bx{y}$ has odd order and set $\SS:=\langle \bx{z}\rangle$. A black box for $\TT_\bx{u}$ can be set up (or a generating set for $\TT_\bx{u}$ can be found) by the arguments in Subsection \ref{sec:centralizerofprotoinvolution}. It follows from the well-known description of subgroups in $\psl_2(\bF)$ that the subgroups $\TT_\bx{u}$ and $\SS$  generate $\YY$.

Consider the amalgam of local proto-automorphisms
\bea
\alpha : \TT_\bx{u} \to \TT_\bx{u}, && \bx{t} \mapsto \bx{t}^{-1}\\
\beta : \SS \to \SS, && \bx{s} \mapsto \bx{s}
\eea
and let $\YY^\circ$ be the resulting involutive proto-automorphism of $\YY$. Note that the automorphism $\bx{d}: (\bx{x},\bx{x}') \mapsto (\bx{x}',\bx{x})$ of $\YY^\circ$ defined by the rule as in Theorem \ref{th:augmentation-by-involution},  so the black box group $\YY^\circ \rtimes \langle \bx{d} \rangle$ encrypts $Y\rtimes \langle d \rangle$. All we need to do now is to make sure that the automorphism $\bx{d}$ encrypts an external involutive automorphism of $Y$.

Observe that if the element $\bx{z}$ belongs to a maximal torus of odd order in $\YY$, then $\bx{d}$ encrypts an external involutive diagonal automorphism of $Y$. However, if the element $\bx{z}$ belongs to a maximal torus of even order in $\YY$, then $\bx{d}$ must be the involution in this torus since $\bx{d}$ centralizes $\SS$. Therefore, $\bx{d}$ encrypts some inner automorphism of $Y$ and we can construct the involution $\bx{d}$ in $\YY$ as a string by constructing the central involution in $\CC_\YY(\bx{d})$.  In this case, we reconstruct a random element $\bx{y} \in \YY$ and repeat our procedure as above. Note that the probability of finding an element $\bx{y} \in \YY$ such that the element $\bx{z}:=\bx{u}\bx{u}^\bx{y}$ belongs to a maximal torus of odd order is at least $1/2$.
\end{proof}

\subsection{Reification of involutions in $\so_3(\bF)$} 

Reification of proto-involutions, as described in Section~\ref{sec:reification}, is the most important procedure involved in our construction of unipotent elements in  $\so_3(\bF)$ and in the proof of Theorem~\ref{SL2-SO3}.

We list some well-known properties of the group $X \simeq \so_3(\bF) \simeq \pgl_2(\bF)$.
\begin{lemma}\label{lm:properties-pgl-2}
Let $X \simeq \so_3(\bF) \simeq \pgl_2(\bF)$ where $\bF$ is a field of odd characteristic. Then:
\bi
\item Every non-trivial element in $X$ is either semisimple or unipotent.
\item Every non-trivial semisimple element $z \in X$ belongs to a unique torus $Z < X$ and centralizes an involution $r \in Z$.
\item If $s$ and $t$ are distinct involutions in $X$ and $z = st$ is a semisimple element, the involution $r$ in the torus  $Z$ containing $z$ is the only involution in $X$ which commutes with both $s$ and $t$.
\ei
\end{lemma}

For two distinct involutions  $\bx{s}, \bx{t} \in \XX$,  we denote the only involution in $\XX$ that commutes with both $\bx{s}$ and $\bx{t}$, if such involution exists, as $\bx{s}\boxtimes \bx{t}$. If it does not exist, then, in view of Lemma~\ref{lm:properties-pgl-2}, the product $\bx{s}\bx{t}$ is a unipotent element.

\begin{theorem}\label{lm:reif}
Let\/ $\XX$ be a black box group encrypting $X=\so_3(\bF)$, where $\bF$ is an unknown finite field of unknown odd characteristic. Assume that $|\bF|\geq 7$ and we know a global exponent $E$ for $\XX$. Let $\bx{s},\bx{t}\in \XX$ be two distinct involutions such that $\bx{s}\bx{t}$ is not a unipotent element. Then there is a Las Vegas algorithm constructing the involution $\bx{j}$ commuting with $\bx{s}$ and $\bx{t}$.

The running time of the algorithm is polynomial in $\log E$.
\end{theorem}

We shall write $\bx{j} =\bx{s}\boxtimes \bx{t}$, treating $\boxtimes$ as a partial binary operation on the set of involutions and call it \emph{cross product}.

\begin{proof} We set $\bx{z}:=\bx{s}\bx{t}$. Note first that if $|\bx{z}|=2$, then $\bx{s}$ and $\bx{t}$ commute and $\bx{z}$ is an involution commuting with them. If $|\bx{z}|>2$, $C_X(\pi_\XX(\bx{z}))$ is a torus containing an involution $\pi_\XX(\jj)$ and it is inverted by $\pi_\XX(\bx{s})$ and $\pi_\XX(\bx{t})$. Clearly, $\bx{j}$ commutes with $\bx{s}$ and $\bx{t}$. By Lemma~\ref{lm:properties-pgl-2}, such an involution $\bx{j}$ is unique in $\XX$.

If the order of $\bx{z}$ is even, then $\bx{j}={\rm i}(\bx{z})$, see Section \ref{sec:centralizerofprotoinvolution}.

If $\bx{z}$ has odd order, then observe that $\bx{j}$ centralizes $\ZZ= \langle \bx{z}\rangle$ and inverts every element in the torus $\TT_\bx{s}$ containing $\bx{s}$; construction of $\TT_\bx{s}$ is similar to construction of tori in the proof of Theorem~\ref{psl2-pgl2}. Since the order of $\bx{z}$ is odd, we have $|\ZZ|\geq 3$ and so $\XX=\langle  \TT_\bx{s},  \ZZ \rangle$. Now the involution $\bx{j}$ can be found by amalgamating local proto-automorphisms
\bea
\bx{x} \mapsto & \bx{x}^{-1} & \mbox{on } \TT_\bx{s}\\
\bx{x} \mapsto & \bx{x} & \mbox{on } \ZZ
\eea
and reifying the result, see Section~\ref{sec:reification}. In $X = \so_3(\bF)$ where $\bF$ is a finite field of odd characteristic with $|\bF| \geq 7$, the last step can be run very efficiently due to the fact that  involutions $r \in X$ have the property that  $Z(C_X(r)) = \langle r \rangle$, see details in  Section~\ref{sec:reification}.
\end{proof}

\subsection{Writing an element in $\so_3(\bF)$ as a product of two involutions}
It is well-known that any element of order $>2$ in $\so_3(\bF)$, where $\bF$ is a finite field of odd characteristic, can be writen as a product of two involutions. The following lemma shows the same can be done in black box groups.

\begin{lemma} \label{lm:bireflectivity}
Let $\XX$ be a black box group encrypting $\so_3(\bF)$, where $\bF$ is a finite field of unknown odd characteristic and $|\bF|\geq 7$.  Then, with a given global exponent $E$ for $\XX$, we can represent an arbitrary element  $\bx{x}\in\XX$ of order $|\bx{x}| >2$ as a product of two involutions $\bx{r}$ and $\bx{s}$ from $\XX$ in time polynomial in $\log E$. In particular, this yields an involution $\bx{r}$ inverting $\bx{x}$. This algorithm is Las Vegas.
\end{lemma}

\begin{proof}
We take a random element $\bx{y}\in \XX$. By \cite{guralnick94.1395}, the probability that $\XX$ is generated by $\bx{x}$ and $\bx{y}$ is at least $1-\frac{1}{2|\bF|^2+2}$. Now, we reify the involution $\bx{r}$ that inverts $\bx{x}$ and $\bx{y}$. If we end up with a failure or a serendipitous discovery of a unipotent element, we need to repeat reification with other choice of $\bx{y}$. When we have the involution $\bx{r}$, we can decompose
$
\bx{x} = \bx{r} \cdot \bx{r}\bx{x},
$
where both $\bx{r}$ and $\bx{s} = \bx{r}\bx{x}$ are involutions.
\end{proof}

\subsection{Unipotency test}

The following lemma makes the algorithm in Theorem \ref{cor:unipotent} Las Vegas.

\begin{lemma}\label{lem:test:uni}
Assume that\/ $\XX$ encrypts $\so_3(\bF)$ where $\bF$ is an unknown finite field of unknown odd characteristic with $|\bF|\geq 7$. Let $\bx{u} \in \XX$  of order bigger than $2$ and $\bx{r}$ an involution inverting $\bx{u}$ -- the latter can be found by Lemma {\rm \ref{lm:bireflectivity}}.
\bi
\item[{\rm (a)}]  Take a random element \/ $\bx{1} \neq \bx{t}\in \CC_\XX(\bx{r})$ with $|\bx{t}|\geq 3$. Then $\bx{u}$ is unipotent if and only if
$\bx{u}\neq \bx{u}^\bx{t}$ and $[\bx{u},\bx{u}^\bx{t}]=\bx{1}$.
\item[{\rm (b)}] Assume that\/ $\bx{u}$ is unipotent. Then $\UU = \langle\bx{u}^{\TT\bx{r}}\rangle$ is the maximal unipotent subgroup in $\XX$ containing $\bx{u}$, $\bx{s}$ inverts $\UU$,  and  $\BB = \UU\TT_\bx{r}$ is the Borel subgroup containing $\UU$.
\ei
The algorithms in this Lemma run in probabilistic time polynomial in $\log E$.
\end{lemma}

\begin{proof} To prove (a), observe that since  $|\bx{t}|\geq 3$, $\bx{t}$ belongs to the torus $\TT_\bx{r}$. Note that the elements of $\XX$ are either semisimple or unipotent. If $\bx{u}$ is semisimple, then $\CC_\XX(\bx{u})$ is the torus $\TT_\bx{u}$. Since $[\bx{u},\bx{u}^\bx{t}]=1$ and $\bx{u} \in \TT_\bx{u}$, we have $\bx{u}^\bx{t} \in \TT_\bx{u}$. Hence $\bx{t} \in \NN_\XX(\TT_\bx{u}) = \TT_\bx{u} \rtimes \langle \bx{w} \rangle$ for some involution $\bx{w}$ inverting $\TT_\bx{u}$. Since $|\bx{t}|\geq 3$, we conclude that $\bx{t} \in \TT_\bx{u}$. This is a contradiction to the assumption $\bx{u}\neq \bx{u}^\bx{t}$. Hence $\bx{u}$ is a unipotent element in $\XX$.

If $\bx{u}$ is unipotent, then part (b) immediately follows from the structural facts about the groups $\so_3(\bF)$ in odd characteristic and also provides the reverse implication in (a).

By the arguments in Section \ref{sec:centralizerofprotoinvolution}, the black box for the subgroup $\TT_\bx{r}$ allows us to construct random conjugates of $\bx{u}$ by the random elements of $\TT_\bx{r}$. Hence we can construct random elements from $\UU$ and $\BB$.
\end{proof}

\subsection{Bisection of angles}\label{sec:bisection}

Note that bisection of angles is the extraction of square roots in the group of rotations.  In our setting, Axiom BB4 allows us to find square roots of elements in cyclic black box subgroups by virtue of the  Tonelli-Shanks algorithm \cite{shanks73.51,tonelli1891.344}.  Usually the  Tonelli-Shanks algorithm is formulated only for multiplicative groups of finite fields and we include its more general formulation for completeness of exposition.

\begin{lemma}[The Tonelli-Shanks Algorithm] \label{lm:Tonelli-Shanks}
Let $\TT$ be a cyclic black box group of known global exponent $E$. Let $\bx{z}$ be an element in $\TT$ that has a square root in $\TT$. Then an element $\bx{t}\in\TT$ such that $\bx{t}^2=\bx{z}$ can be found in probabilistic polynomial time in $\log E$.
\end{lemma}

\begin{proof}
We set $E=2^mn$ where $(2,n)=1$. Given $\bx{t}\in \TT$, we shall say that $l$ is the $2$-height of $\bx{t}$, if $|\bx{t}^n| = 2^l$; notice that this is equivalent to $2^l$ being the largest power of $2$ that divides the order $|\bx{t}|$ of $\bx{t}$.

We first construct an element $\bx{x}\in \TT$ with maximal 2-height $l$, that is, the order $|\bx{x}|$ is divisible by the maximum power of 2 dividing the order of $\TT$. To do this, we construct a constant number of random elements in $\TT$ and take the element with the biggest 2-height. Note that since $\TT$ is cyclic, at least half of the elements of $\TT$ have orders with maximal 2-height. If a chosen element does not have the biggest 2-height, then the procedure below fails, and we start our procedure by constructing an element with bigger 2-height. Let $\bx{x}\in \TT$ be an element with maximal 2-height $l$, that is, the order $\bx{x}$ is divisible by the maximum power of 2 dividing the order of $\TT$. If $l\neq 0$, then $\bx{x}$ can not be a square in $\TT$, namely, there are no elements $\bx{y} \in \TT$ such that $\bx{y}^2=\bx{x}$. We set
 $$\bx{a}:=\bx{z}^{(n+1)/2},\quad \bx{b}:=\bx{z}^n,\quad \bx{c}:=\bx{x}^n.$$ Note that if $\bx{b}=1$, then $\bx{z}$ has odd order and the element $\bx{a}$ is the desired square root of $\bx{z}$. If $\bx{b}\neq 1$, then we run the loop:
  \bi
\item Find the smallest non-negative integer $d$ such that $\bx{b}^{2^d}=1$.
\item  If $d>0$ then repeat until $d=0$: $$\mbox{Set}\quad \bx{a}:=\bx{a}\bx{c}^{2^{l-d-1}},\quad \bx{b}:=\bx{b}\bx{c}^{2^l-d},\quad \bx{c}:=\bx{c}^{2^{l-d}},\quad l:=d.$$
\item When $d=0$, the element $\bx{a}$ is the desired square root of $\bx{z}$.
    \ei
\end{proof}

\begin{lemma} \label{lm:bisection}
Let $\XX$ be a black box group encrypting $\so_3(\bF)$, where $\bF$ is an unknown finite field of unknown odd characteristic. Assume that $|\bF|\geq 7$ and $\bx{i},\bx{j} \in \XX$ are two conjugate involutions. Then, given an exponent $E$ for $\XX$, we can find an involution $\bx{x}\in \XX$ such that $\bx{i}^\bx{x}=\bx{j}$ in time polynomial in $\log E$.
\end{lemma}

\begin{proof}
We set $E=2^mn$ where $(2,n)=1$, and set $\bx{z}=\bx{i}\bx{j}$.

If the order of $\bx{z}$ is odd, that is, $\bx{z}^n=1$ then notice that $\bx{i}^{\bx{z}^{(n+1)/2}}=\bx{j}$. Now, ${\bx{z}^{(n+1)/2}}\bx{j}$ is an involution conjugating $\bx{i}$ to $\bx{j}$.

Assume now that the order of $\bx{z}$ is even and $\bx{k}=\bx{i}(\bx{z})$ is the involution in $\langle \bx{z} \rangle$, see Section~\ref{sec:centralizerofprotoinvolution}. We denote by $\YY$ the subgroup in $\XX$ encrypting $\psl_2(\bF)$; it is well-known that $|\XX:\YY| =2$ and $\YY \triangleleft \XX$.

Let $\TT$ be the maximal torus in $\XX$ containing $\bx{k}$ and $\TT^2 = \{\bx{t}^2 \mid \bx{t}\in \TT\}$, then $\TT^2$ is the subgroup of index $2$ in $\TT$ and $\TT^2 = \TT \cap \YY$. Observe that  $\bx{z} = \bx{i}\bx{j} \in \TT^2$ because $\bx{i}$ and $\bx{j}$, being conjugate, simultaneously belong or do not belong to $\YY$.

We can now apply the Tonelli-Shanks algorithm for cyclic groups, Lemma~\ref{lm:Tonelli-Shanks}, and find $\bx{t} \in \TT$ such that $\bx{t}^2=\bx{z}=\bx{ij}$; after that we have
\[
\bx{i}^{\bx{t}\bx{j}}=\bx{j}\bx{t}^{-1}\bx{i}\bx{t}\bx{j}=\bx{j}\bx{t}^{-1}\bx{j}\bx{j}\bx{i}\bx{t}\bx{j}=\bx{t}\bx{j}\bx{i}\bx{t}\bx{j}=\bx{t}^2\bx{j}\bx{i}\bx{j}=\bx{j}, \]
and $\bx{x} = \bx{tj}$ is an involution since $\bx{j}$ inverts $\bx{t} \in \TT$.
\end{proof}

\section{Geometry of involutions in $\pgl_2(\bF)\simeq \so_3(\bF)$ }\label{sec:geometryofinvolutions}

Let $X = \pgl_2(\bF)\simeq\so_3(\bF)$, where $\bF$ is a finite field of odd characteristic. In this section, we see that actions of involutions from $X$ control properties of every facet of the structure of the group and its Lie algebra. Involutions are multifunctional: they act as pointers to tori in the group $X$, to toric subalgebras in the Lie algebra $\lie =\Lie(X)$ of $X$, to points and to lines in the projective plane associated with $\lie$ as $\bF$-vector space, and they control the canonical polarity on $P$.

\subsection{The Lie algebra}

Let $M$ be the $2\times 2$ matrix algebra over $\bF$. We denote the elements of $M$ by lower case Greek letters.

The Lie algebra $\lie=\lsl_2(\bF)$ of the group $X = \pgl_2(\bF)\simeq\so_3(\bF)$  is the vector space of $2\times 2$ matrices of trace $0$ with the Lie bracket $[\alpha,\beta] = \alpha\beta - \beta\alpha$ and the Killing form $K(\alpha,\beta) = \Tr(\alpha\beta)$. The isomorphism $\pgl_2(\bF)\simeq\so_3(\bF)$ comes from the adjoint action of $\gl_2(\bF)$ on   the Lie algebra $\lie$, that is, action by conjugation on $\lie$.  The group $\pgl_2(\bF)$ is the image of this action and becomes the group of automorphisms of $\lie$. Therefore it preserves the Killing form $K$ on $\lie$,
moreover, it coincides with the orthogonal group $\so_3(\lie,K)$.

The following property of matrices in $M$ can be easily checked.

\begin{lemma} \label{Lie-elementary1} Let $\alpha \in M$ be a non-scalar matrix. Then $\alpha^2$ is a scalar matrix if and only if $\Tr(\alpha) =0$.
\end{lemma}

Similarly, the following properties of $\lie = \lsl_2(\bF)$ can also be easily proven.

\begin{lemma} \label{Lie-elementary2} Let $\alpha,\beta \in \lie$, that is, $\Tr(\alpha) = \Tr(\beta) =0$.
\bi

\item[{\rm (a)}] $\alpha$ is either non-degenerate and semisimple  or $\alpha^2=0$.
\item[{\rm (b)}] $\Tr(\alpha^2) = -\det \alpha$. As a consequence, $\Tr(\alpha^2) =0$ if and only if $\alpha^2=0$.
\item[{\rm (c)}] $[\alpha,\beta]=0$ if and only if $\alpha$ and $\beta$ are collinear, that is, one of them is a scalar multiple of another one.
\item[{\rm (d)}] $\Tr(\alpha, [\alpha,\beta]) = 0$.
\ei
\end{lemma}

\subsection{Projectivization of the Lie algebra: polarity} \label{sec:projectivization}

Now consider the projective space $\bP(M)$ associated with the vector space $M$ and the natural map
\[
\omega: M \smallsetminus\{0\} \longrightarrow \bP(M).
\]
Notice that $X = \pgl_2(\bF)$ is the image of $\gl_2(\bF)$ under this map.

Denote by $P=\omega(\lie)$ the image of $\lie$ in $\bP(M)$; it is a projective plane with polarity $\pi$ induced by the Killing form on $\lie$. It follows from Lemmas \ref{Lie-elementary1} and \ref{Lie-elementary2}(a)  that
$I =P \cap X$ is the set of involutions in $X$. By Lemma \ref{Lie-elementary2}(b), the set $Q = P\smallsetminus I$ is a quadric in $P$ determined by the quadratic form $\Tr(\alpha^2)$ ($=-\det\alpha$ by Lemma~\ref{Lie-elementary2}(b)).

The group $X$ has a natural action  on $\bP(M)$ induced by the action of $\gl_2(\bF)$ on $M$ by conjugation.

Obviously, $X$ leaves invariant the projective plane $P$ and its subsets $I$ and $Q$. Moreover $X$ acts on $P$ by collineations and preserves the polarity $\pi$.

There are two kinds of points in $P$:
\bi
\item \emph{involutive} (or \emph{toric}, or \emph{semisimple}, or \emph{regular}) -- they form the set $I$;
\item \emph{unipotent} (or \emph{parabolic}, or \emph{tangent}) -- they are points on the quadric $Q$.
\ei

We shall call $I$ the \emph{involution plane}; it is a projective plane with a polarity, but with points of the corresponding quadric removed.

For points $a, b \in P$ we shall write $a \perp b$ if $a \in \pi(b)$ (which is equivalent to $b \in \pi(a)$). If $\alpha,\beta \in \lie$ are matrices representing $a$ and $b$, respectively, then $a\perp b$ is the same as $K(\alpha,\beta) = \Tr(\alpha\beta) = 0$.

Notice also that
\[
Q = \{\, a \in P : a \in \pi(a)\,\}.
\]

\subsection{Projectivization of the Lie algebra: cross product} \label{sec:cross-product}

The Lie product $[\,,\,]$ on $\lie$ induces a partial binary operation on the plane $P$: namely, if $a$ and $b$ are distinct points in $P$ represented by matrices $\alpha,\beta\in \lie$, respectively, then, in view of  Lemma \ref{Lie-elementary2}(c), $[\alpha,\beta] \ne 0$ and $\omega([\alpha,\beta])$ is defined as in Section  \ref{sec:projectivization} and does not depend on choice of $\alpha$ and $\beta$. We denote
\[
\omega([\alpha,\beta]) = a \boxtimes b
\]
and call it the \emph{cross product} of $a$ and $b$. Obviously, $a \boxtimes b = b \boxtimes a$. The cross product has an obvious connection with the polarity:

\begin{lemma}
If $a$ and $b$ are distinct points in $P$ then
\[
a \perp a \boxtimes b \;\mbox{ and }\; b \perp a \boxtimes b,
\]
and, consequently,
\[
a \boxtimes b = \pi(a) \wedge \pi(b)
\]
is the point of intersection of the polar lines of $a$ and $b$.
\end{lemma}

\begin{proof}
The proof follows from Lemma \ref{Lie-elementary2}(d).
\end{proof}

Notice that action of $X$ on $P$ preserves the cross product. We use the term ``cross product'' and notation $\boxtimes$ to emphasize the analogy with the cross product of vectors in $\bR^3$, that is, the Lie algebra operation in the Lie algebra of $\so_3(\bR)$, the group of rotations of the 3-dimensional Euclidean space.

The projective plane $P$ with polarity and cross product is a combinatorial object that retains essential properties of the Lie algebra $\lie = \lsl_2$; we shall call it the \emph{projectivization} of $\lie$.
As we shall soon see, polarity and cross product can be constructed inside of the group $X$: first on the set $I$ of involutions and then extended to the whole plane $P$ by interpreting the points of the quadric $Q$ as maximal unipotent subgroups of $X$. Moreover, these constructions can be carried out in a black box group $\XX$ encrypting $X$.

\subsection{Points} \label{sec:involutions} 

Now we turn our attention to involutions in $X$ which will serve as points in our projective plane.

In view of Lemma \ref{Lie-elementary2}(b), a  vector $\sigma \in \lie$ is semisimple in the Lie algebra sense if and only if $K(\sigma,\sigma) \neq 0$ (that is, $\sigma$ is regular in the   terminology of the theory of quadratic forms) and nilpotent if and only if $K(\sigma,\sigma)=0$ (that is, $\sigma$ is singular).

Every semisimple element $\sigma$ in $\lie$ gives rise to an involution in $X$, the half-turn $s_\sigma$ around the one-dimensional vector subspace (it is also a Lie subalgebra) generated by $\sigma$:
\begin{equation}\label{eq:halfturn}
s_\sigma: \alpha \mapsto \frac{2K(\alpha,\sigma)}{K(\sigma,\sigma)}\sigma -\alpha.
\end{equation}

Observe that the half-turn $s_\sigma$ is not changed if we replace $\sigma$ by a non-zero scalar multiple $c\sigma$.

\begin{lemma}
\label{involutions=half-turns}
An involution $s = \omega(\sigma)$ in $X$ represented by a matrix $\sigma$ is the half-turn around $\sigma$:
\[
\omega(\sigma) = s_\sigma.
\]
\end{lemma}

\begin{proof}
Indeed, in its adjoint action on  $\lie$, every involution $s$ from $X$ has eigenvalues $+1, -1, -1$. If $s =\omega(\sigma)$ then obviously $\sigma^s = \sigma$ is an eigenvector for $s$ and eigenvalue  $+1$. Obviously this means that then $s = s_\sigma$.
\end{proof}

Denote the  $+1$-eigenspace (the \emph{axis} of the half-turn) $s$ as $\tor_s$. Obviously, $\tor_s$ is a $1$-dimensional regular  subspace of $\lie$ and thus a Cartan subalgebra of $\lie$. If $T_s$ is a torus in $X$ containing $s$ then $\tor_s = \Lie(T_s)$, the Lie algebra of $T_s$.

\begin{lemma}\label{lm:perp}
Let $i,j \in I$ and $i\ne j$, then $i \perp j$ if and only if $ij=ji$.
\end{lemma}

\begin{proof}
By Lemma~\ref{involutions=half-turns}, $i$ and $j$ are half-turns. Now it easily follows from Equation \ref{eq:halfturn} in Section \ref{sec:involutions} that they commute if and only if their axes are orthogonal, that is, if and only if $i \perp j$.
\end{proof}

Lemma \ref{lm:perp} allows us to interpret the polarity restricted to the involution plane $I$ within the group $X$:

\begin{lemma}
If $i \in I$ then
\[
\pi(i)\cap I = \{\, j\in I \mid ij = ji \mbox{ and } i \ne j\,\}.
\]
\end{lemma}

We shall denote
\[
\pi_I(i) = \pi(i) \cap I.
\]

Similarly, we have the following result for the cross product:

\begin{lemma} \label{lm:cross=commute}
If $i,j \in I$ are distinct involutions and $k =i\boxtimes j \in I$ then
\[
k = \pi(i) \wedge \pi(j)
\]
is the only involution in $X$ which commutes with the both $i$ and $j$.
\end{lemma}

Later in the paper, we shall extend the  polarity and cross product to the whole projective plane $P = I \cup Q$.

\subsection{Lines} 

For every involution $s \in I$, its polar image $\pi(s)$ in $P$ can be described by taking its intersection with $I$: $\pi_I(s)= \pi(s) \cap I$.  It could happen that one or two points in $\pi(s) \cap Q$ are missing from $\pi_I(s)$. Set $q = |\bF|$. The centralizer $C=C_X(s)$ is a dihedral group of order $2(q\pm 1)$  which contains the maximal torus $T_s$ of order $q\pm 1$  inverted by $q\pm 1$ involutions in the coset $C \smallsetminus T_s$; these involutions commute with $s$ and therefore
\[
\pi_I(s) = C\smallsetminus T_s.
\]
As we shall see in the next section, there are more lines on $P$.

\subsection{The Weisfeiler plane}

Every $1$-dimensional subspace $\lia$ in $\lie$ is a Lie subalgebra of $\lie$ and coincides with the Lie algebra $\Lie(A)$ of some $1$-dimensional algebraic subgroup $A<X$. Assuming that $|\bF|=q$, the latter belongs to one of the three conjugacy classes:
\bi
\item non-split tori: cyclic subgroups of order $q+1$,
\item  maximal unipotent subgroups of order $q$,
\item split tori: cyclic subgroups of order $q-1$,
\ei
see the paper by Boris  Weisfeiler \cite{weisfeiler79.522}.

Therefore the set $W$ of $1$-dimensional algebraic subgroups $A$ in $X$ is in one-to-one correspondence $$A \leftrightarrow \Lie(A)$$ with the set of $1$-dimensional Lie subalgebras of $\lie$ and can be treated as a projective plane; we shall call \emph{the Weisfeiler plane} and denote it by $W$.

It will be convenient to identify $W$ with the dual plane $P^*$ of $P$ and treat points of $W$ as \emph{lines} of $P$. For that we need to describe the incidence relation between points and lines.

If $A$ is $1$-dimensional subgroup in $X$, the line $\ell(A)$ associated with it contains (incident with) all involutions inverting $A$; if $w$ is one of these involutions then $\ell_I = \ell(A)\cap I$ coincides with  the coset $Aw$.

\bi
\item If $A$ is a non-split torus, all $q+1$ points in $\ell(A)$ are involutions.
\item  If $A$ is a maximal unipotent subgroup, $q$ points in $\ell(A)$ are involutions in the Borel subgroup $N_X(A)$ inverting $A$, and the additional parabolic point in $\ell(A) \cap C$ is $\omega(\Lie(A))$ for $A$ itself.
\item If $A$ is a split torus, $q-1$ points in $\ell(A)$ are involutions inverting $A$; two additional parabolic points in $\ell(A)\cap Q$ are $\omega(\Lie(U))$ and $\omega(\Lie(V))$ for two maximal unipotent subgroups $U$ and $V$ normalized by $A$.
\ei

These three types of lines are called \emph{hyperbolic},  \emph{parabolic}, and  \emph{elliptic}, respectively. The parabolic lines are tangent lines to $Q$, that is, lines having exactly one point with $Q$ in common. In $I$, a parabolic line appears as the coset $U t$ of a maximal unipotent subgroup $U$ in $X$ with respect to an involution $t$ inverting every element in $U$.

\section{The black box projective plane and projectivization of the Lie algebra}\label{sec:projective-plane}

Let $\XX$ be a black box group encrypting $X=\pgl_2(\bF) \simeq \so_3(\bF)$ where $\bF$ is a finite field of odd characteristic and $|\bF|\geq 7$.

Using the black box $\XX$ as a computational engine, we shall construct a black box projective plane $\Pp$ that encrypts the projective plane $P$ discussed in Section \ref{sec:geometryofinvolutions}.

The elements or objects of $\Pp$ are \emph{points} and \emph{lines}.

\subsection{Points} There are two types of points in $\Pp$; \emph{regular} and \emph{parabolic}.

A \emph{regular point} is an involution in $\XX$; $\Ii$ is the set of all   involutions in $\XX$. To produce a random regular point, we construct an involution from a random element in $\XX$ by repeated square and multiply method and conjugate it by a random element. Note that a random element  in $\XX$ has even order with probability at least 1/4 \cite{isaacs95.139}.

A \emph{parabolic point}   a black box for a maximal unipotent subgroup $\UU<\XX$}; it is a point on the quadric $\Qq=\Pp\smallsetminus \Ii$; to construct one is the principal aim of the paper.

\subsection{Lines}\label{sec:lines}  

There are two types of lines in $\Pp$; \emph{toric} and \emph{parabolic}.

A \emph{toric} or \emph{regular line} $\ll$ is a black box for a subgroup $\TT\rtimes\langle \bx{w}\rangle$ where $\TT < \XX$ is a torus  and $\bx{w}\in \XX$ is an involution that inverts every element in $\TT$. A toric line is incident to the following points:
\bi
\item If $|\TT|=q+1$ then $\ll$ is incident only to points represented by involutions in the coset $\TT \bx{w}$;
\item If $|\TT|=q-1$ then $\ll$ is incident to $q-1$ points represented by involutions in the coset $\TT \bx{w}$ and, in addition,  two parabolic points which will be constructed later and correspond to two maximal unipotent subgroups normalized by $\TT$.
\ei
It is convenient to specify a toric line using its polar image, that is, the involution in the torus $\TT$.

A \emph{parabolic line} (or \emph{tangent} line) $\uu$ is a black box for a subgroup $\UU\rtimes\langle \bx{t}\rangle$ where $\UU<\XX$ is a maximal unipotent subgroup and $\bx{t}\in \XX$ is an involution inverting every element in $\UU$ -- it exists and can be computed by Lemma~\ref{lem:test:uni}(b). The line $\uu$ is incident to two kinds of points:
\bi
\item $q$ regular points, involutions in the coset $\UU \bx{t}$; and
\item  $\UU$ itself, seen as a point.
\ei

\subsection{Cross product on $\Ii$ and a serendipity construction of a unipotent element} \label{sec:serendipity-parabolic}

As we can see, we have immediate access only to points in $\Ii$, that is, involutions in $\XX$. Our principal tool is the cross product provided by Theorem \ref{lm:reif}.

In view of Lemma \ref{lm:cross=commute}, the operation $\boxtimes$ on $\Ii$ encrypts the cross product on $I$.

For large $q$, the probability of hitting a unipotent element $\bx{u}$ by taking product of two random involutions $\bx{s}$ and $\bx{t}$ is $O\left(1/q\right)$, as can be seen from the following argument.
Indeed, note that the number of involutions in $\XX$ is $q^2$ where $q=|\bF|$ and for a fixed involution $\bx{s}$, there are at most two unipotent subgroups normalized by $\bx{s}$. Therefore,
the number of unipotent elements of the form $\bx{s}\bx{s}^\bx{x}$, $\bx{x}\in \XX$,   is at most $2(q-1)$. Therefore, the probability of this event is approximately   $O\left(1/q\right)$ and becomes astronomically small for a large field $\bF$.

However if it happens by a sheer strike of luck,  we get a unipotent element $\bx{u}=\bx{s}\bx{t}$ and a black box for the parabolic subgroup
\[
\BB = \langle \bx{u}^{\TT_\bx{s}}\rangle\rtimes \TT_\bx{s},
\]
its maximal unipotent subgroup
\[
\UU = \langle \bx{u}^{\TT_\bx{s}}\rangle,
\]
(that is, a parabolic point)
and the set $\UU \bx{s}$ of regular points in the corresponding parabolic line, see Lemma~\ref{lem:test:uni}.

Combining Theorem \ref{lm:reif} with the arguments above, we have
\begin{theorem} \label{th:cross-with-serendipity} There is probabilistic polynomial time Las Vegas algorithm which, given two involutions $\bx{i}, \bx{j} \in \Ii$, constructs either
\bi
\item an involution $\bx{i} \boxtimes \bx{j}$, or
\item a black box subgroup $\UU\rtimes\langle\bx{i}\rangle = \UU\rtimes\langle\bx{j}\rangle$ where $\UU$ is  a maximal unipotent subgroup in $\XX$ inverted by $\bx{i}$ and $\bx{j}$. The black box for $\UU$ represents a parabolic point in $\Pp$ which coincides with $\bx{i}\boxtimes\bx{j}$ in sense of Section {\rm \ref{sec:cross-product}}.
\ei
\end{theorem}

The strategy of our proof of Theorem~\ref{cor:unipotent} is to carry out a polynomial time chain of constructions in $\Ii$ which will force the serendipity moment, that is, the second bullet point in Theorem \ref{th:cross-with-serendipity}. In the next section, we shall describe tools for computations in $\Ii$ needed for that purpose.

\subsection{A toolbox for the involution plane $\Ii$}

Constructions in this section are conditional on assumption that $\bx{s} \boxtimes \bx{t}\in \Ii$ is an involution for all involutions $\bx{s}$ and $\bx{t}$ that we encounter in our calculations -- as it was explained in Section \ref{sec:serendipity-parabolic}, this is what normally happens with very high probability.

We shall expand our algorithms to the whole black box projective plane $\Pp$ in Section \ref{sec:expansion-to-Pp}.

\subsubsection{Polar image of a point in $\Ii$} Let $q = |\bF|$.

Let $\bx{t}$ be an involution and $\TT_\bx{t}$ a torus containing it, and $\bx{w}$ an involution inverting $\TT_\bx{t}$.
\bi
\item If $|\TT_\bx{t}| = q+1$ then the polar line $\bxg{\pi}(\bx{t})$ is the coset $\TT_\bx{t}\bx{w}$, with random points in it generated in an obvious way with the help of a black box for $\TT_\bx{t} < \CC_\XX(\bx{t})$.

\item If $|\TT_\bx{t}| = q-1$ then the polar line $\bxg{\pi}(\bx{t})$ is the coset $\TT_\bx{t}\bx{w}$ together with two points represented by black boxes for maximal unipotent subgroups  $\UU$  and $\VV$  normalized by $\TT_\bx{t}$. At this stage we do not know how to construct $\UU$ and $\VV$, but it will become clear after Theorem \ref{cor:unipotent} is proven: we take a nontrivial unipotent element $\bx{u}$, construct the maximal unipotent subgroup $\AA$ and an involution $\bx{s}$ inverting $\AA$ as in Lemma~\ref{lem:test:uni}, and then construct an element $\bx{x}$ conjugating the involution $\bx{s}^{\bx{x}} = \bx{t}$. Now $\TT_\bx{t}$ normalizes $\UU = \AA^{\bx{x}}$, and if $\bx{r} \in \CC_\XX(\bx{t}) \smallsetminus \TT_\bx{t}$ then $\TT_\bx{t}$ normalizes $\UU^{\bx{r}} = \VV$ as well.
\ei

\subsubsection{Polar image of a line in $\Ii$}

Let $\kk$ be a line in $\Pp$. Given a line, we can always find on it two distinct regular points, see Section \ref{sec:lines}; take points $\bx{a} \ne \bx{b}$ on $\kk$.
Then
\[
\bxg{\pi}(\kk) = \bx{a} \boxtimes \bx{b}.
\]

\subsubsection{A line through two regular points}\label{sec:lines2}

If  $\bx{s},\bx{t}\in \Ii$ then the line  $\bx{s}\vee \bx{t}$ through  $\bx{s}$ and $\bx{t}$ is
\[
 \bx{s}\vee \bx{t} = \bxg{\pi}(\bx{s}\boxtimes \bx{t}).
\]

We shall note here that we do not list the points on the black box projective lines but we produce random elements on them when they are needed.

\subsubsection{Intersection of two lines in $\Ii$}\label{sec:int:lines}

If two lines $\kk$ and $\ll$ intersect in $\Ii$, their intersection point $\kk \wedge \ll$ can be found as
\[
\kk \wedge \ll = \bxg{\pi}\left( \bxg{\pi}(\kk) \vee \bxg{\pi}(\ll)  \right).
\]

\subsection{Expansion of the toolbox from $\Ii$ to $\Pp$} \label{sec:expansion-to-Pp}

\subsubsection{Polarity in $\Pp$}

If $\UU$ is a maximal unipotent subgroup in $\XX$ seen as a parabolic point and $\bx{t}$ an involution inverting $\UU$ which can be constructed by Lemmas \ref{lm:bireflectivity} and \ref{lem:test:uni}, then $\bxg{\pi}(\bx{\UU})$ is the tangent line
$\UU \bx{t} \cup \{\UU\}$ through $\UU$.

\subsubsection{Intersection of two lines} 

Since we have now polarity in $\Pp$, the formula is the same as in \ref{sec:int:lines}.

\subsubsection{Cross product in $\Pp$}

For two distinct points  $\bx{s}, \bx{t} \in \Pp$, 
\bea
\bx{s}\boxtimes \bx{t} &=& \bxg{\pi}(\bx{s}) \wedge \bxg{\pi}(\bx{t}). \\
\eea

\subsubsection{A line through parabolic points} 

Let $\bx{s}$ be an involution and $\UU$  a parabolic point, that is, a maximal unipotent subgroup. Observe that if $\bx{s}$ inverts $\UU$ then $\bx{s}$ belongs to the line tangent to $\Qq$ at $\UU$. In this case, we have
\[
\bx{s}\vee \UU = \UU \cup \UU\bx{s}.
\]
Assume now that $\bx{s}$ does not invert $\UU$. In this case, the line $\bx{s} \vee \UU$ is the polar line of the involution $\bx{r}$ which centralizes $\bx{s}$ and inverts $\UU$. Since $\bx{r} \in \CC_\XX(\bx{s})$, $\bx{r}$ inverts every element in the torus $\TT_{\bx{s}}$. Therefore, $\bx{r}$ can be reified from the amalgamation of the the following proto-morphisms:
\bea
\TT_{\bx{s}} \to \TT_{\bx{s}}, && \bx{t} \mapsto \bx{t}^{-1}\\
\UU \to \UU, && \bx{u} \mapsto \bx{u}^{-1}.
\eea
Moreover, since $\bx{s}$ and $\UU$ belong to the line $\bx{s}\vee \UU$ and $\bx{s}$ does not normalize $\UU$, the other missing point in this line is the unipotent group $\UU^{\bx{s}}$. Hence, in this case, we have
\[
\bx{s}\vee \UU = \TT_\bx{r}\bx{s} \cup \{\UU, \UU^{\bx{s}}\}.
\]

For the line passing through given two parabolic points $\UU$ and $\VV$, we follow the similar procedure as above. Note that, in this case, the line $\UU \vee \VV$  is the polar line of the involution $\bx{r}$ which inverts both $\UU$ and $\VV$. Therefore, $\bx{r}$ can be reified similarly and
\[
\UU\vee \VV = \TT_\bx{r} \bx{w} \cup\{\UU, \VV\}
\]
for some $\bx{w} \in \CC_\XX(\bx{r})$ inverting $\TT_\bx{r}$.

\section{Construction of $\Sym_4$}\label{sec:sym4}

It is well-known that there is only one conjugacy class of subgroups isomorphic to $\Sym_4$ in $\so_3(\bF)$ over a finite field of odd characteristic. The fundamental procedure in the coordinatization of $\Pp$ is the construction of a black box subgroup encrypting $\Sym_4$ in a black box group encrypting $\so_3(\bF)$ over a finite field of odd characteristic. As we shall soon see, a subgroup isomorphic to $\Sym_4$  provides us with a convenient basis triangle in $\Pp$.

\begin{theorem}\label{sym4:in:pgl2}
Let $\XX$ be a black box group encrypting $X=\so_3(\bF)$ over an unknown finite field $\bF$ of unknown odd characteristic and $|\bF|\geq 7$. Then, given a global exponent $E$ for $\XX$, there is a polynomial in $\log E$ time Las Vegas algorithm constructing a black box
subgroup in $\XX$ which encrypts a subgroup in $X$ isomorphic to $\Sym_4$.
\end{theorem}

We precede our proof of Theorem~\ref{sym4:in:pgl2} with a few lemmas. We work within the terminological conventions of Section~\ref{sec:blackboxsubgroups} and apply to strings and black box subgroups of $\XX$ the same terms as to corresponding elements and subgroups of $X$.

We work under assumptions of Theorem~\ref{sym4:in:pgl2}. It is well-known that $\XX$ has two conjugacy classes of involutions. We say that an involution is of \emph{$+$-type} if the order of its centralizer is $2(q-1)$ and \emph{$-$-type} if the order of its centralizer is $2(q+1)$. Notice that $\CC_\XX(\bx{i})=\TT_\bx{i}\rtimes \langle \bx{w} \rangle$ where $\TT_\bx{i}$ is a torus of order $(q\pm 1)$ and $\bx{w}$ is an involution inverting $\TT_\bx{i}$.  We will consider the involutions of $+$-type if $\qpone$ and $-$-type if $\qmone$ so that the order of the corresponding  torus is always divisible by $4$; we will call them involutions of \emph{right type}.

We are looking for a $5$-tuple
\begin{equation*}\label{set:1}
(\bx{i},\bx{j},\bx{z},\bx{s},\TT_\bx{i})
\end{equation*}
where $\bx{i}\in \XX$ is an involution of right type, $\bx{j}\in \XX$ is an involution of right type which inverts  $\TT_\bx{i}$, $\bx{z}\in \XX$ is an element of order $3$ normalizing $\langle \bx{i},\bx{j}\rangle$ and $\bx{s} \in \TT_\bx{i}$ is an element of order $4$. We also set $\bx{k}=\bx{i}\bx{j}$ and note that $\bx{k}$ is also of right type. Clearly $\langle \bx{i},\bx{j},\bx{z}\rangle$ encrypts a subgroup isomorphic to $\Alt_4$ and  $\langle \bx{i},\bx{j},\bx{z},\bx{s} \rangle$ encrypts $\Sym_4$.
The crucial part of our construction is the search an element $\bx{z}\in \XX$ of order 3 permuting some mutually commuting involutions $\bx{i},\bx{j},\bx{k} \in \XX$ of right type. The following lemma provides explicit construction of such an element.

\begin{lemma}\label{elt:3}
 Let $\bx{i},\bx{j},\bx{k} \in \XX$ be mutually commuting involutions of right type and $\bx{x} \in \XX$ be an arbitrary element. Assume that $\bx{y}_1=\bx{i}\bx{j}^\bx{x}$ has odd order $m_1$ and set $\bx{n}_1=\bx{y}_1^{\frac{m_1+1}{2}}$ and $\bx{s}=\bx{k}^{\bx{x}\bx{n}_1^{-1}}$. Assume also that $\bx{y}_2=\bx{j}\bx{s}$ has odd order $m_2$ and set  $\bx{n}_2=\bx{y}_2^{\frac{m_2+1}{2}}$. Then  the element $\bx{z}=\bx{x}\bx{n}_1^{-1}\bx{n}_2^{-1}$ permutes $\bx{i},\bx{j},\bx{k}$, and $\bx{z}$ has order $3$.
\end{lemma}

\begin{proof}
Observe first that $\bx{i}^{\bx{n}_1}=\bx{j}^\bx{x}$ and $\bx{j}^{\bx{n}_2}=\bx{s}$. Since $\bx{s}=\bx{k}^{\bx{x}\bx{n}_1^{-1}}$, we have $\bx{j}^{\bx{n}_2}=\bx{k}^{\bx{x}\bx{n}_1^{-1}}$. Hence $\bx{j}=\bx{k}^{\bx{x}\bx{n}_1^{-1}\bx{n}_2^{-1}}=\bx{k}^\bx{z}$. Now, we prove that $\bx{j}^\bx{z}=\bx{i}$. Since $\bx{j}^{\bx{x}\bx{n}_1^{-1}}=\bx{i}$, we have $\bx{j}^{\bx{z}}=\bx{j}^{\bx{x}\bx{n}_1^{-1}\bx{n}_2^{-1}}=\bx{i}^{\bx{n}_2^{-1}}$. We claim that $\bx{y}_2 \in \CC_\XX(\bx{i})$, which implies that $\bx{n}_2 \in \CC_\XX(\bx{i})$, so $\bx{j}^\bx{z}=\bx{i}^{\bx{n}_2^{-1}}=\bx{i}$. Now, since $\bx{j} \in \CC_\XX(\bx{i})$, $\bx{y}_2=\bx{j}\bx{s}\in \CC_\XX(\bx{i})$ if and only if $\bx{s} =\bx{k}^{\bx{x}\bx{n}_1^{-1}}\in \CC_\XX(\bx{i})$. Moreover, since $\bx{i}^{\bx{n}_1}=\bx{j}^\bx{x}$,  $\bx{s} \in \CC_\XX(\bx{i})$ if and only if $\bx{k}^\bx{x} \in \CC_\XX(\bx{j}^\bx{x})$, equivalently, $\bx{k} \in \CC_\XX(\bx{j})$ and the claim follows. It is now clear that $\bx{i}^\bx{z}=\bx{k}$ since $\bx{i}\bx{j}=\bx{k}$, and $\bx{z}$ has order $3$.
\end{proof}

\begin{lemma}\label{elt:3:prob}
Let $\XX$, $\bx{i},\bx{j},\bx{k},\bx{y}_1,\bx{y}_2$ and $\bx{z}$ be as in Lemma {\rm \ref{elt:3}}. Then the probability that $\bx{y}_1$ and $\bx{y}_2$ have odd orders is bounded from below by $\frac{1}{2}-\frac{1}{2|\bF|}$.
\end{lemma}

\begin{proof}
We first note that the subgroup $\langle \bx{i},\bx{z}\rangle \cong \Alt_4$ is a subgroup of $\YY\leq \XX$ where $\YY\cong \psl_2(\bF)$, so the involutions $\bx{i},\bx{j},\bx{k}$ belong to a normal subgroup isomorphic to $ \psl_2(\bF)$. Therefore it is enough to compute the estimate in $\YY$. Notice that all involutions in $\YY$ are conjugate. Therefore the probability that $\bx{y}_1$ and $\bx{y}_2$ have odd orders is the same as the probability of the product of two random involutions from $\YY$ to be of odd order.

The rest of computation is done in the underlying group $Y=\pi(\YY)$. We set $|\bF|=q$ and  we denote by $a$ one of these numbers $(q\pm 1)/2$ which is odd and by $b$ the other one. Then $|Y|=q(q^2-1)/2=2abq$ and $|C_Y(i)|=2b$ for any involution $i\in Y$. Hence the total number of involutions is
\[
\frac{|Y|}{|C_Y(i)|}=\frac{2abq}{2b}=aq.
\]

Now we compute the number of pairs of involutions $(i,j)$ such that their product $ij$ belongs to a torus of order $a$. Let $T$ be a torus of order $a$. Then $N_Y(T)$ is a dihedral group of order $2a$. Therefore the involutions in $N_Y(T)$ form the coset $N_Y(T)\backslash T$ since $a$ is odd. Hence, for every torus of order $a$, we have $a^2$ pairs of involutions whose product belong to $T$. The number of tori of order $a$ is $|Y|/|N_Y(T)|=2abq/2a=bq$. Hence, there are $bqa^2$ pairs of involutions whose product belong to a torus of order $a$. Thus the desired probability is
$$
\frac{bqa^2}{(aq)^2}=\frac{b}{q}\geqslant \frac{q-1}{2q}=\frac{1}{2}-\frac{1}{2q}.
$$
\end{proof}

\begin{proof}[Proof of Theorem \ref{sym4:in:pgl2}]
Let $E=2^mn$ where $(2,n)=1$. We first construct an involution $\bx{i}\in \XX$ of right type and an element $\bx{s} \in \CC_\XX(\bx{i})$ of order $4$. Let $\bx{i}\in \XX$ be an involution constructed from a random element by taking its power using square-and-multiply method. To check whether $\bx{i}$ is an involution of right type or not, we search for an element $\bx{s} \in \CC:=\CC_\XX(\bx{i})$ of order 4. Note that a random element from $\CC$ can be constructed efficiently by the arguments in Section \ref{sec:centralizerofprotoinvolution}. Note also that if $\bx{i}$ is of right type then $\CC$ contains elements of order 4, otherwise, $\CC$ does not contain elements of order 4. If $\bx{i}$ is of right type then, since $\CC = \TT_\bx{i}\rtimes \langle \bx{w} \rangle$, where $\TT_\bx{i}$ is a torus of order $q \pm 1$ and $\bx{w}$ is an involution which inverts $\TT_{\bx{i}}$,  a random element from $\CC$ has order divisible by 4 with probability at least $1/4$. As soon as we find an element $\bx{y}\in \CC$ such that $\bx{y}^n\neq 1$ and $\bx{y}^{2n}\neq 1$, then we construct an element $\bx{s} \in \langle \bx{y} \rangle$ of order $4$ by repeated square-and-multiple method. If we can not find an element of order $4$ in $\CC$, we deduce that $\bx{i}$ is not of right type and we start from the beginning.

Let $\bx{i}\in \XX$ be a right type involution. The coset $\TT_\bx{i} \bx{w}$ of $\TT_\bx{i}$ in $\CC$ consists of the involutions inverting $\TT_\bx{i}$, so half of the elements of $\CC$ are the involutions inverting $\TT_\bx{i}$ and half of the involutions in $\TT_\bx{i} \bx{w}$ are of the same type as $\bx{i}$. We construct an involution $\bx{j}\in \CC$ and check whether $\bx{j}$ is an involution of right type by following the same arguments above.

Finally, for commuting right type involutions $\bx{i},\bx{j}\in \XX$, we construct an element $\bx{z}$ of order 3 normalizing $\langle \bx{i},\bx{j}\rangle$ by using Lemma \ref{elt:3}. The probability of constructing such an element $\bx{z}\in \XX$ is at least $\frac{1}{2}-\frac{1}{2|\bF|}$ by Lemma \ref{elt:3:prob}. Hence $\langle \bx{s},\bx{z} \rangle$ is a black box subgroup encrypting $\Sym_4$.
\end{proof}

\section{Coordinatization and a construction of a black box field} \label{sec:coordinatisation}

To construct a black box field in $\XX$, all we need is to carry out Hilbert's coordinatization of $\Pp$ \cite{hartshorne67} using our toolbox from Section~\ref{sec:projective-plane}.

\subsection{The spinor basis}

A construction from Section~\ref{sec:sym4} yields a black box subgroup $\HH$ encrypting $\Sym_4$ and we shall need to introduce special notation for some of its elements as they will play the central role in later calculations.

We denote the three involutions in the $4$-group $\EE=O_2(\HH)$ by $\bx{e}_1,\bx{e}_2,\bx{e}_3$. If $\tor_1, \tor_2, \tor_3$ are the centralizers in the Lie algebra $\lie$ of their images $\pi(\bx{e}_1),\pi(\bx{e}_2),\pi(\bx{e}_3)$, respectively, we know that they are orthogonal to each other and
\[
\lie = \tor_1 \oplus \tor_2 \oplus \tor_3
\]
is the weight decomposition for the action of $\EE$ on $\lie$ and is therefore a grading of $\lie$:
\[
[\tor_1, \tor_2] = \tor_3,\quad [\tor_2, \tor_3] = \tor_1,\quad [\tor_3, \tor_1] = \tor_2.
\]
Moreover, an element $\bx{z}$ of order $3$ from $\HH$ cyclically permutes $\tor_1, \tor_2, \tor_3$, which allows us to select a basis in $\lie$ made of
\[
\epsilon_1 \in \tor_1, \quad \epsilon_2 = \epsilon_1^\bx{z} \in \tor_2, \;\mbox{ and }\; \epsilon_3 = \epsilon_2^\bx{z} \in \tor_3.
\]
Since $\EE$ lies in the commutator of $\HH$, the involutions $\bx{e}_i\in \EE$ have spinor norm $1$ and therefore vectors $\epsilon_i$ can be chosen to satisfy
\[
K(\epsilon_i,\epsilon_i) = 1
\]
forming an orthonormal basis in $\lie$,
\[
K(\epsilon_i, \epsilon_j) = \delta_{ij}.
\]
In particular, the quadric $\Qq$ in $\Pp$  can be written by the equation
\[
\bx{x}_1^2+\bx{x}_2^2+\bx{x}_3^2 =0
\]
in the coordinates $\bx{x}_1,\bx{x}_2,\bx{x}_3$ associated with the basis $\epsilon_1, \epsilon_2,\epsilon_3$.

In addition, the basis $\epsilon_1, \epsilon_2,\epsilon_3$ seen as a basis of the Lie algebra $\lie$  obviously satisfies the Lie relations
\[
[\epsilon_1, \epsilon_2] = a\epsilon_3, \quad  [\epsilon_2, \epsilon_3] = a\epsilon_1, \quad [\epsilon_3, \epsilon_1] = a\epsilon_2,
\]
for some fixed $a \in \bF_q^*$. What we found is an analogue of a \emph{spinor basis} (or \emph{Pauli basis}) from quantum mechanics and we will discuss these in detail elsewhere.

\subsection{First steps towards the coordinatization of $\Pp$}

We know that $\epsilon_1, \epsilon_2,\epsilon_3$ form an orthonormal basis in $\lie$ and  $\bx{e}_1,\bx{e}_2,\bx{e}_3$ have
homogeneous coordinates $$(1,0,0), (0,1,0), (0,0,1);$$
and the quadric $\Qq$ is given in  coordinates $\bx{x}_1,\bx{x}_2,\bx{x}_3$ associated with this basis by the equation
\[
\bx{x}_1^2+\bx{x}_2^2+\bx{x}_3^2=0.
\]
Following traditional notation, we represent lines in $\Pp$ by equations of the form
\[
\XX_1\bx{x}_1 +\XX_2\bx{x}_2+\XX_3\bx{x}_3 = 0
\]
and treat the tuple $[\XX_1,\XX_2,\XX_3]$ as the homogeneous coordinates of the line.

\subsection{First steps in construction of a black box field} \label{sec:constructionBBfield}

We shall now construct a black box field $\KK$. Towards this end, we take the set of points on the line $\bx{e}_1 \vee \bx{e}_3$ for the extended field $\KK \cup\{\bxg{\infty}\}$ by assigning the coordinate $\bx{x}_1=\bx{0}$ to $\bx{e}_3$ and $\bx{x}_1=\bxg{\infty}$ to $\bx{e}_1$. We call the line $\bx{e}_1 \vee \bx{e}_3$ the $\bx{x}_1$-axis and similarly the line $\bx{e}_2 \vee \bx{e}_3$ the $\bx{x}_2$-axis.

Taking into account that the coordinatization of $\Pp$ has to be consistent with the action of $\XX$, and, in particular, with the action of $\HH$ on the basis $\bx{e}_1,\bx{e}_2,\bx{e}_3$, we see that if we take the line $\bx{e}_1 \vee \bx{e}_2$ for the line at infinity, we have the following:

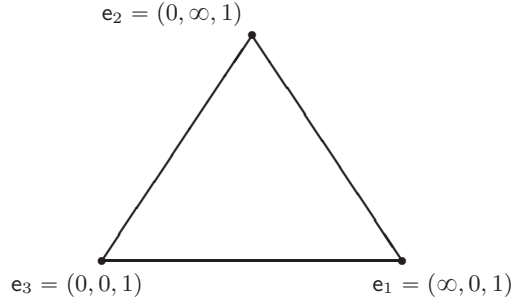
\begin{center}
\setlength{\unitlength}{40mm}

\begin{picture}(1.5,1.2)(-.25, -.25)

\thicklines

\put(0,0){\circle*{0.03}}
\put(.5,.75){\circle*{0.03}}
\put(1,0){\circle*{0.03}}
\put(0,0){\line(1,0){1}}
\put(0,0){\line(2,3){.5}}
\put(.5,.75){\line(2,-3){0.5}}
\put(-.3,-.1){{\small $\bx{e}_3=(0,0,1)$}}
\put(.9,-.1){{\small $\bx{e}_1=(\bxg{\infty},0,1)$}}
\put(0,.8){{\small $\bx{e}_2=(0,\bxg{\infty},1)$}}

\thinlines

\end{picture}
\end{center}

\noindent
Here, $0$ and $1$ can be seen as elements of our future black box field $\KK$.

The following is the same picture in homogeneous coordinates:

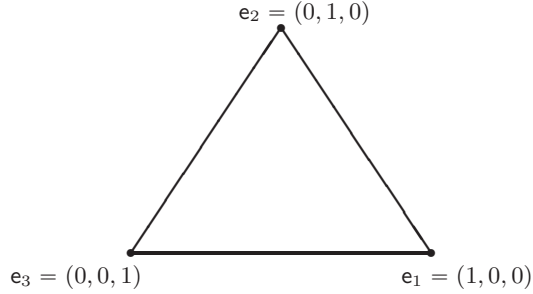
\begin{center}
\setlength{\unitlength}{40mm}

\begin{picture}(1.5,1.2)(-.25, -.25)

\thicklines

\put(0,0){\circle*{0.03}}
\put(.5,.75){\circle*{0.03}}
\put(1,0){\circle*{0.03}}
\put(0,0){\line(1,0){1}}
\put(0,0){\line(2,3){.5}}
\put(.5,.75){\line(2,-3){0.5}}
\put(-.4,-.1){{\small $\bx{e}_3=(0,0,1)$}}
\put(.9,-.1){{\small $\bx{e}_1=(1,0,0)$}}
\put(.36,.78){{\small $\bx{e}_2=(0,1,0)$}}

\end{picture}
\end{center}

We shall gradually assign coordinates to more and more points in $\Pp$, at every step ensuring that the coordinatization is consistent with the action of $\XX$ on $\Ii$ and $\Pp$ and hence with the vector space structure on $\lie$.
If a point $\bx{x}\in \Pp$ has coordinates $\bx{x}_1,\bx{x}_2, \bx{x}_3$, we shall write
\[
\bx{x}= (\bx{x}_1,\bx{x}_2,\bx{x}_3).
\]
Similarly, we denote lines by their coordinates,
\[
{\bx{\ell}} = [\XX_1,\XX_2,\XX_3]
\]
which denote the line
\[
{\bx{\ell}} = \{\, (\bx{x}_1,\bx{x}_2,\bx{x}_3) \mid X_1\bx{x}_1 +X_2\bx{x}_2+X_3\bx{x}_3 = 0\,\}.
\]
We note  that $(\bx{x}_1,\bx{x}_2,\bx{x}_3)$ and $[\XX_1,\XX_2,\XX_3]$ are homogeneous coordinates, they are defined up to multiplication by a non-zero scalar.

Observe that polarity has a very simple meaning in terms of homogeneous coordinates associated with an orthonormal basis:
\[
\polar((\bx{x}_1,\bx{x}_2,\bx{x}_3)) = [\XX_1,\XX_2,\XX_3] \;\mbox{ if and only if }\; \XX_1 = \bx{x}_1,\; \XX_2 = \bx{x}_2,\; \XX_3 = \bx{x}_3.
\]
In particular, polar images of the base points $\epsilon_i$ have equations $\bx{x}_i=0$, $i=1,2,3$, and homogeneous coordinates
\[
\polar(\epsilon_1) = [1,0,0],\quad \polar(\epsilon_2) = [0,1,0],\quad \polar(\epsilon_3) = [0,0,1].
\]

So we have, in the black box setup, the following picture.

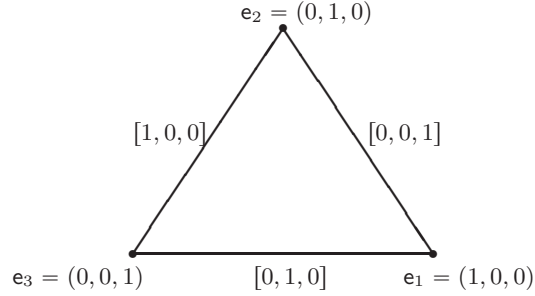
\begin{center}
\setlength{\unitlength}{40mm}

\begin{picture}(1.5,1.2)(-.25, -.25)

\thicklines

\put(0,0){\circle*{0.03}}
\put(.5,.75){\circle*{0.03}}
\put(1,0){\circle*{0.03}}
\put(0,0){\line(1,0){1}}
\put(0,0){\line(2,3){.5}}
\put(.5,.75){\line(2,-3){0.5}}
\put(-.4,-.1){{\small $\bx{e}_3=(0,0,1)$}}
\put(.9,-.1){{\small $\bx{e}_1=(1,0,0)$}}
\put(.36,.78){{\small $\bx{e}_2=(0,1,0)$}}
\put(.78,.38){{\small $[0,0,1]$}}
\put(.4,-.1){{\small $[0,1,0]$}}
\put(0,.38){{\small $[1,0,0]$}}
\end{picture}
\end{center}

We shall soon add new points to this picture.

\subsection{The unity element in $\KK$}\label{subsec:unity}

So far, we know which elements on the $\bx{x}_1$-axis represent the point $0$ and $\bxg{\infty}$ and now we construct the unity element on the $x_1$-axis.

Let $\bx{z}$ be an element of order $3$ in $\HH$ which permutes the basis points $\bx{e}_1,\bx{e}_2,\bx{e}_3$. Pick in $N_\HH(\langle{\bx{z}}\rangle)$ an involution $\bx{d}_1$ which commutes with $\bx{e}_1$. Observe that $\EE\rtimes \langle \bx{d}_1\rangle$ is a dihedral group of order $8$ and therefore $\bx{e}_2^{\bx{d}_1}=\bx{e}_3$.

Now turn to the use of homogeneous coordinates. Recall that $\bx{e}_2 = (0,1,0)$ and $\bx{e}_3 = (0,0,1)$. There are two involutions which conjugate $\bx{e}_2$ and $\bx{e}_3$ (see Equation \ref{eq:halfturn} in Section \ref{sec:involutions}):
\bea
s_{(0,1,1)}(\bx{e}_2) &=& s_{(0,1,1)}((0,1,0))\\ &=& \frac{2(0\cdot 0+1\cdot 1+1\cdot 0)}{0^2+1^2+1^2}(0,1,1) - (0,1,0)\\ &=&  (0,0,1)\\ &=& \bx{e}_3
\eea
and
\bea
\quad\; s_{(0,1,-1)}(\bx{e}_2) &=& s_{(0,1,-1)}((0,1,0))\\ &=& \frac{2(0\cdot 0+1\cdot 1+(-1)\cdot 0)}{0^2+1^2+1^2}(0,1,-1) - (0,1,0)\\ &=&  (0,0,-1)\\ &=& \bx{e}_3
\eea

We can assign to $\bx{d}_1$ the coordinates $(0,1,1)$ and set
\[
\bx{d}_2 = \bx{d}_1^{\bx{z}} = (1,0,1) \mbox{ and } \bx{d}_3 = \bx{d}_1^{\bx{z}^2} = (1,1,0).
\]
So we have now a richer picture:

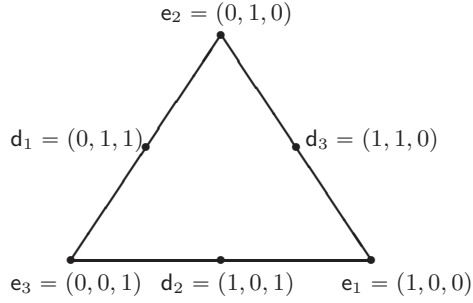
\begin{center}
\setlength{\unitlength}{40mm}

\begin{picture}(1.5,1.2)(-.25, -.25)

\thicklines

\put(0,0){\circle*{0.03}}
\put(.5,.75){\circle*{0.03}}
\put(1,0){\circle*{0.03}}
\put(0,0){\line(1,0){1}}
\put(0,0){\line(2,3){.5}}
\put(.5,.75){\line(2,-3){0.5}}
\put(-.2,-.1){{\small $\bx{e}_3=(0,0,1)$}}
\put(.9,-.1){{\small $\bx{e}_1=(1,0,0)$}}
\put(.3,.8){{\small $\bx{e}_2=(0,1,0)$}}

\put(.75,.375){\circle*{0.03}}
\put(.78,.38){{\small $\bx{d}_3=(1,1,0)$}}
\put(.5,0){\circle*{0.03}}
\put(.3,-.1){{\small $\bx{d}_2=(1,0,1)$}}
\put(.25,.375){\circle*{0.03}}
\put(-0.2,.38){{\small $\bx{d}_1=(0,1,1)$}}
\end{picture}
\end{center}

\subsection{More about $\HH$}

We record for future use that the natural isomorphism
\[
\HH \longrightarrow \Sym_4,
\]
where $\Sym_4$ is seen as the symmetric group of the set $\{\,0,1,2,3\,\}$ in notation chosen in such a way that
\[
\begin{array}{lll}
\bx{e}_1 \mapsto (01)(23) & \quad {\bx{z}} \mapsto (123) \quad & \bx{d}_1 \mapsto (23)\\
\bx{e}_2 \mapsto (02)(13) &   & \bx{d}_2 \mapsto (13)\\
\bx{e}_3 \mapsto (03)(12) &   & \bx{d}_3 \mapsto (12)
\end{array}.
\]
In particular,
\[
\bx{d}_2^{\bx{d}_3}= \bx{d}_1, \quad \bx{e}_1^{\bx{d}_3} = \bx{e}_2.
\]

\subsection{Affine coordinates} \label{ss:coordinates}

Taking, as we have already did, the line $\bx{e}_1\vee \bx{e}_2$ for the line at infinity and the lines $\bx{x}_2=0$ and $\bx{x}_1=0$ for the coordinate axes, we get

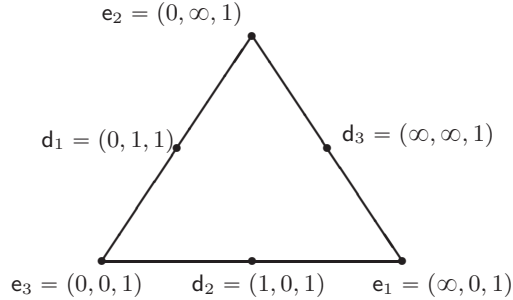
\begin{center}
\setlength{\unitlength}{40mm}

\begin{picture}(1.5,1.2)(-.25, -.25)

\thicklines

\put(0,0){\circle*{0.03}}
\put(.5,.75){\circle*{0.03}}
\put(1,0){\circle*{0.03}}
\put(0,0){\line(1,0){1}}
\put(0,0){\line(2,3){.5}}
\put(.5,.75){\line(2,-3){0.5}}
\put(-.3,-.1){{\small $\bx{e}_3=(0,0,1)$}}
\put(.9,-.1){{\small $\bx{e}_1=(\bxg{\infty},0,1)$}}
\put(0,.8){{\small $\bx{e}_2=(0,\bxg{\infty},1)$}}

\put(.75,.375){\circle*{0.03}}
\put(.8,.4){{\small $\bx{d}_3=(\bxg{\infty},\bxg{\infty},1)$}}
\put(.5,0){\circle*{0.03}}
\put(.3,-.1){{\small $\bx{d}_2=(1,0,1)$}}
\put(.25,.375){\circle*{0.03}}
\put(-0.2,.38){{\small $\bx{d}_1=(0,1,1)$}}

\thinlines
\end{picture}
\end{center}

Observe that this assignment of coordinates agrees with action by $\HH$. In particular, conjugation by $\bx{d}_3$ moves the points $0$, $1$, $\bxg{\infty}$ on the $\bx{x}_1$-axis to the points $0$, $1$, $\bxg{\infty}$ on the $\bx{x}_2$-axis, respectively. Therefore we can treat both coordinate axes as the two copies of the projective line $\KK \cup\{\bxg{\infty}\}$ over the black box field $\KK$ that we will construct on the $\bx{x}_1$-axis.

Now on ``this side of infinity'', on the  affine plane $\bx{x}_3\ne 0$, the homogeneous  coordinates of arbitrary point $\bx{x}$ can be written as
$(\bx{x}_1,\bx{x}_2, 1)$, where
\[
\bx{x}_1 = (\bx{x} \vee \bx{e}_2) \wedge (\bx{e}_1 \vee \bx{e}_3) \;\mbox{ and }\; \bx{x}_2 = (\bx{x} \vee \bx{e}_1) \wedge (\bx{e}_2 \vee \bx{e}_3)
\]
are projections of $\bx{x}$ onto the coordinate axes, and we get the classical coordinatization of the affine plane \cite{coxeter2003projective}:

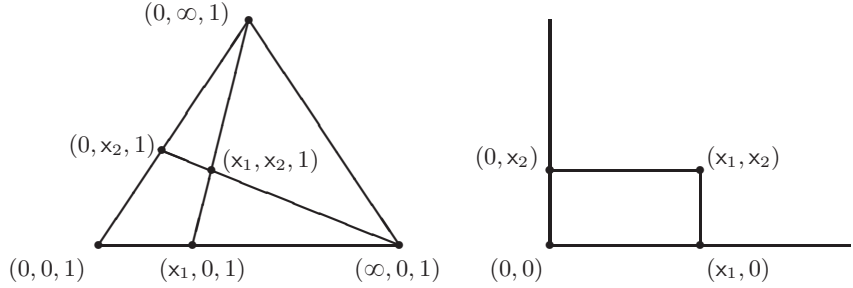
\begin{center}
\setlength{\unitlength}{40mm}

\begin{picture}(5,1.2)(-.4, -.25)

\thicklines

\put(0,0){\circle*{0.03}}
\put(.5,.75){\circle*{0.03}}
\put(1,0){\circle*{0.03}}
\put(1,0){\line(-5,2){0.8}}
\put(0,0){\line(1,0){1}}
\put(0,0){\line(2,3){.5}}
\put(.5,.75){\line(2,-3){0.5}}
\put(.5,.75){\line(-1,-4){0.187}}
\put(.3125,0){\circle*{0.03}}
\put(.2,-0.1){{\small $(\bx{x}_1,0,1)$}}
\put(-.3,-.1){{\small $(0,0,1)$}}
\put(.85,-.1){{\small $(\bxg{\infty},0,1)$}}
\put(0.15,.75){{\small $(0,\bxg{\infty},1)$}}
\put(.375,.25){\circle*{0.03}}
\put(.41,.26){{\small $(\bx{x}_1,\bx{x}_2,1)$}}
\put(.21,.3175){\circle*{0.03}}
\put(-.1,.31){{\small $(0,\bx{x}_2,1)$}}

\put(1.5,0){\circle*{0.03}}
\put(1.5,0){\line(1,0){1}}
\put(1.5,0){\line(0,1){.75}}
\put(1.5,.25){\circle*{0.03}}
\put(1.5,.25){\circle*{0.03}}
\put(1.5,.25){\line(1,0){.5}}
\put(2,.0){\line(0,1){.25}}
\put(2,.0){\circle*{0.03}}
\put(2.02,-.1){{\small $(\bx{x}_1,0)$}}
\put(2,.25){\circle*{0.03}}
\put(2.02,.27){{\small $(\bx{x}_1,\bx{x}_2)$}}
\put(1.25,.27){{\small $(0,\bx{x}_2)$}}
\put(1.3,-.1){{\small $(0,0)$}}
\end{picture}
\end{center}

If $\bx{x}$ lies on the line at infinity $\bx{x}_3 =0$ then we can take any point $\bx{x}'$ on the line $\bx{e}_3 \vee \bx{x}$, construct its affine coordinates $(\bx{x}'_1,\bx{x}'_2,1)$ as above and take the triple $(\bx{x}'_1,\bx{x}'_2,0)$ for the homogeneous coordinates of $\bx{x}$.

\subsection{Addition $\oplus$  on $\KK$} \label{sec:add-in-K}

Now we can introduce the field operations in the usual way, as shown on the following two diagrams, see Hartshorne \cite{hartshorne67} for details. We do this on the $\bx{x}_1$-axis. Note that the set of points on this axis consists of involutions in $\CC_\XX(\bx{e}_2)$ except $\bx{e}_2$ together with the two parabolic points -- the two maximal unipotent subgroups inverted by $\bx{e}_2$ -- if the involution $\bx{e}_2$ is of the $+$-type in the sense of Section \ref{sec:sym4}.

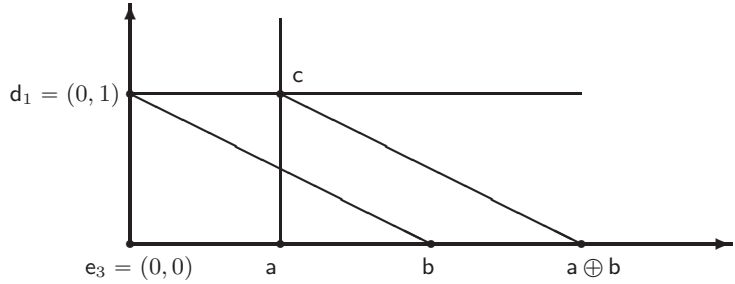
\setlength{\unitlength}{40mm}

\begin{picture}(5,1.25)(-.3, -.25)

\thicklines

\put(0,0){\circle*{0.03}}
\put(0,0){\vector(1,0){2}}
\put(-.15,-.1){{\small $\bx{e}_3 = (0,0)$}}
\put(0,0){\vector(0,1){.8}}
\put(-.4,.475){{\small $\bx{d}_1 =(0,1)$}}
\put(0,.5){\circle*{0.03}}
\put(0,.5){\line(1,0){1.5}}
\put(.5,0){\line(0,1){.75}}
\put(.5,.5){\circle*{0.03}}
\put(.54,.54){{\small $\bx{c}$}}
\put(.5,0){\circle*{.03}}
\put(.45,-.1){{\small $\bx{a}$}}
\put(0,.5){\line(2,-1){1}}
\put(1,0){\circle*{.03}}
\put(.97,-.1){{\small $\bx{b}$}}
\put(.5,.5){\line(2,-1){1}}
\put(1.5,0){\circle*{.03}}
\put(1.45,-.1){{\small $\bx{a}\oplus \bx{b}$}}
\end{picture}

In terms of our toolbox, we first construct
\[
\bx{c} = (\bx{a} \vee \bx{e}_2) \wedge (\bx{d}_1 \vee \bx{e}_1),
\]
then we construct the point at infinity on the line $\bx{d}_1 \vee \bx{b}$ and denote it $\bxg{\infty}_{\bx{d}_1, \bx{b}}$:
\[
\bxg{\infty}_{\bx{d}_1,\bx{b}} = (\bx{d}_1 \vee \bx{b}) \wedge (\bx{e}_1 \vee \bx{e}_2),
\]
then $\bx{a}\oplus \bx{b}$ is the point of intersection of the line $\bx{c} \vee \bx{d}$ parallel to $\bx{d}_1 \vee \bx{b}$ with the $\bx{x}_1$-axis $\bx{e}_1 \vee \bx{e}_3$:
\[
\bx{a}\oplus \bx{b} = (\bx{c} \vee \bxg{\infty}_{\bx{d}_1,\bx{b}}) \wedge (\bx{e}_1 \vee \bx{e}_3).
\]

\subsection{Multiplication $\otimes$ on $\KK$} \label{sec:multiplication-in-K}

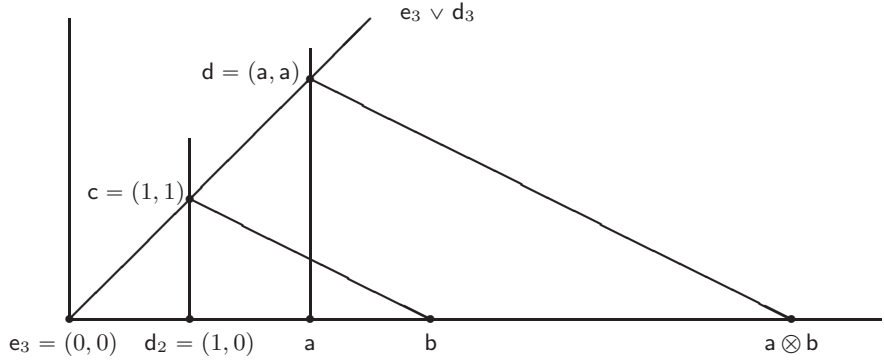
\begin{center}
\setlength{\unitlength}{40mm}

\begin{picture}(5,1.25)(-.3, -.25)

\thicklines

\put(0,0){\circle*{0.03}}
\put(0,0){\line(1,0){2.7}}
\put(-.2,-.1){{\small $\bx{e}_3=(0,0)$}}
\put(0,0){\line(0,1){1}}
\put(0,0){\line(1,1){1}}
\put(1.1,1){{\small $\bx{e}_3 \vee \bx{d}_3$}}
\put(.4,0){\circle*{.03}}
\put(.25,-.1){{\small $\bx{d}_2=(1,0)$}}
\put(.4,0){\line(0,1){.6}}
\put(.06,.4){{\small $\bx{c}=(1,1)$}}
\put(.4,.4){\circle*{.03}}
\put(.8,0){\circle*{.03}}
\put(.78,-.1){{\small $\bx{a}$}}
\put(.8,0){\line(0,1){.9}}
\put(.44,.8){{\small $\bx{d}=(\bx{a},\bx{a})$}}
\put(.8,.8){\circle*{.03}}
\put(1.2,0){\circle*{.03}}
\put(.4,.4){\line(2,-1){.8}}
\put(1.18,-.1){{\small $\bx{b}$}}
\put(.8,.8){\line(2,-1){1.6}}
\put(2.4,0){\circle*{.03}}
\put(2.31,-.1){{\small $\bx{a}\otimes \bx{b}$}}
\end{picture}
\end{center}

In terms of our toolbox, we first construct the line $\bx{x}_1=\bx{x}_2$ as $\bx{e}_3\vee \bx{d}_3$, then the point $\bx{c}=(1,1)$ as
\[
(\bx{e}_3\vee \bx{d}_3)\wedge (\bx{d}_2 \vee \bx{e}_2),
\]
and point $\bx{d} = (\bx{a},\bx{a})$ as
\[
\bx{d}= (\bx{e}_3\vee \bx{d}_3) \wedge (\bx{a} \vee \bx{e}_2),
\]
then the point at infinity of the line $\bx{b} \vee \bx{c}$ as
\[
\bxg{\infty}_{\bx{b},\bx{c}} =(\bx{b} \vee \bx{c}) \wedge (\bx{e}_1 \vee \bx{e}_2),
\]
the line through the point $\bx{d}$ parallel to $\bx{b} \vee \bx{c}$ as
\[
\bx{d} \vee \bxg{\infty}_{\bx{b},\bx{c}},
\]
and, finally, the product $\bx{a}\otimes \bx{b}$ as the point of intersection of that line with the $\bx{x}_1$-axis $\bx{e}_1 \vee \bx{e}_3$:
\[
\bx{a}\otimes \bx{b} = (\bx{e}_1 \vee \bx{e}_3) \wedge (\bx{d} \vee \bxg{\infty}_{\bx{b},\bx{c}}).
\]

\subsection{Inversion and negation in $\KK$}

Forming the negative
\[
\bx{x} \mapsto \mathop{\footnotesize \ominus}\bx{x}
\]
and inversion
\[
\bx{x} \mapsto \bx{x}^{\ominus}
\]
on $\KK$ are much easier to compute than addition and multiplication.  Here are two useful observations.

If $\bx{x} = ({\bxg{\chi}},0,1)$ is a point in the $\bx{x}_1$-axis,
\bea
s_{(0,0,1)}(\bx{x}) &=& s_{(0,0,1)}(({\bxg{\chi}},0,1))\\
    &=& \frac{2(0\cdot {\bxg{\chi}}+ 0\cdot 0 +1\cdot 1)}{0^2+0^2+1^2}(0,0,1) - ({\bxg{\chi}},0,1)\\
    &=& (0,0,2) - ({\bxg{\chi}},0,1)\\
     &=& (-{\bxg{\chi}}, 0,1)\\
     &=& \mathop{\footnotesize \ominus} \bx{x}
\eea
and
\bea
s_{(1,0,1)}(\bx{x}) &=& s_{(1,0,1)}(({\bxg{\chi}},0,1))\\
    &=& \frac{2(1\cdot {\bxg{\chi}}+ 0\cdot 0 +1\cdot 1)}{1^2+0^2+1^2}(1,0,1) - ({\bxg{\chi}},0,1)\\
    &=& ({\bxg{\chi}}+1,0,{\bxg{\chi}}+1) - ({\bxg{\chi}},0,1)\\
    &=& (1,0,{\bxg{\chi}})\\
    &=& (1/{\bxg{\chi}}, 0,1)\\
    &=& \bx{x}^{\ominus}.
\eea
Therefore the field operations of taking negative and inversion
\[
\bx{x} \mapsto \mathop{\footnotesize \ominus} \bx{x},\quad \bx{x} \mapsto \bx{x}^{\ominus}
\]
on $\KK$ are computable by single conjugations, that is, for a regular point $\bx{x}$ in $\bxg{\pi}(\bx{e}_2)$,
\[
\ominus \bx{x} = \bx{x}^{\bx{e}_3}  \;\mbox{  and }\; \bx{x}^{\ominus}=\bx{x}^{\bx{d}_2}.
\]
This completes the construction of the black box field $\KK$.

\subsection{Square roots in $\KK$} \label{ss:squareroot}

Given an element $\bx{x} \in \KK$,  a number of polynomial time Las Vegas algorithms allow us to find a square root of $\bx{x}$ in $\KK$, if it exists. In our context, the Tonelli-Shanks algorithm, Lemma~\ref{lm:Tonelli-Shanks}, is suitable for our purposes since the multiplicative group $\KK^*$ of $\KK$ is isomorphic to a torus in $\XX$ and inherits the global exponent from $\XX$.

\section{Enforced serendipity: construction of unipotent elements} \label{sec:serendipity}

The aim of this section is to prove Theorem \ref{cor:unipotent}. We start by presenting a test which decides whether an element in $\pgl_2(\bF)\simeq \so_3(\bF)$ is unipotent or not when the characteristic of the field $\bF$ is not known. This allows us to make our algorithm presented in Theorem \ref{cor:unipotent} a Las Vegas algorithm.

We will use the following lemma to locate a unipotent element in $\XX$.

\begin{lemma}\label{lem:quad:res}
Let $\qmone$. Then, for random $x,y \in \bF_q$, $-x^2-y^2$ is a non-zero quadratic residue with probability $1/2-1/2q^2$.
\end{lemma}

\begin{proof}
By \cite[Theorem 10.5.1]{berndt1998}, for a fixed $0\neq c \in \bF_q$, the number of solutions of the equation $-x^2-y^2=c^2$  is $q+1$. Since there are $(q-1)/2$ non-zero elements in $\bF_q$ which are quadratic residue, $-x^2-y^2$ is a quadratic residue with probability
\[
\frac{(q+1)(q-1)}{2q^2}=\frac{q^2-1}{2q^2}= \frac{1}{2}- \frac{1}{2q^2}
\]
for random $x,y \in \bF_q$. Hence the result follows.
\end{proof}

\begin{proof}[Proof of Theorem \ref{cor:unipotent}]

Let $\YY$ be a black box group encrypting $\psl_2(\bF)$ over some unknown field of unknown odd characteristic $p$. Let $E$ be an exponent for $\YY$ and $E=2^mn$, $(2,n)=1$.

First, we construct a black box group $\XX$ encrypting $\so_3(\bF)$ from the given black box group $\YY$ by using Theorem \ref{psl2-pgl2}. Next, we construct three commuting involutions $\bx{e}_1, \bx{e}_2, \bx{e}_3$ in $\YY$ which we take for an orthogonal basis in our projective plane and the  starting points for its coordinatization. Then we construct a black box subgroup $\HH < \XX$ encrypting $\Sym_4$ containing $\bx{e}_1, \bx{e}_2, \bx{e}_3$ (Theorem~\ref{sym4:in:pgl2}). Note that $\bx{e}_1, \bx{e}_2, \bx{e}_3$, being involutions in $\YY$, are right type involutions in $\XX$ in the sense of Section \ref{sec:sym4}. Finally, by following the procedures described in Section \ref{sec:coordinatisation}, we have a black box field $\KK$ with addition, $\oplus$, and multiplication, $\otimes$, together with the procedures for computing multiplicative and additive inverses. Let $\KK$ be defined on the $\bx{x}_1$-axis, that is, the elements of $\KK \cup \{\bxg{\infty} \}$ are the involutions in the coset
\[
\TT_{\bx{e}_2}\bx{e}_1 = \TT_{\bx{e}_2}\bx{e}_3,
\]
together with two parabolic points --  the maximal unipotent subgroups normalized by $\bx{e}_2$ --  if the involution $\bx{e}_2$  is of $+$-type in sense of Section \ref{sec:sym4}, and we choose the involutions $\bx{e}_3$ and $\bx{e}_1$ for the roles of\/ $\bx{0}$ and $\bxg{\infty}$, respectively.

We work in the affine plane $\bx{x}_3 = \bx{1}$ in $\Pp$ as constructed in Section~\ref{ss:coordinates} with coordinates $(\bx{x}_1,\bx{x}_2, \bx{1})$. In these coordinates, the  quadratic equation which defines the conic $\Qq$ is $\bx{x}_1^2 \oplus \bx{x}_2^2 \oplus \bx{1} =\bx{0}$.

Let $\bx{d}_1$ and $\bx{d}_2$ be the unit elements on the $\bx{x}_2$-axis  and $\bx{x}_1$-axis, respectively, constructed as described in Subsection \ref{subsec:unity}. Let $\bx{d}_3$ be the unit element that moves the $\bx{x}_1$-axis to the $\bx{x}_2$-axis coordinate-wise, see Subsection \ref{ss:coordinates}.

Now, we distinguish the cases $\qpone$ and $\qmone$ to construct a unipotent element in $\XX$.

If $\qpone$, then the coset $\TT_{\bx{e}_2} \bx{e}_3$ has $q-1$ involutions. Therefore, $\KK$ has two missing points which are precisely the parabolic points on the $\bx{x}_1$-axis. Since the $\bx{x}_2$-coordinate of any point on $\bx{x}_1$-axis is $\bx{0}$, in this case, there exists an element $\bx{c} \in \KK$ such that $\bx{c}^2\oplus \bx{1}=\bx{0}$. This means that, on the $\bx{x}_1$-axis, the points with the homogeneous coordinates $(\pm \bx{c},\bx{0},\bx{1})$ lie on the conic $\Qq$ so they are parabolic and the construction of one of these parabolic points gives a unipotent element in the black box group $\XX$. Observe that the multiplicative order of the elements $\pm \bx{c} \in \KK$ is 4. Therefore, the construction of an element of order 4 in the multiplicative group of $\KK$ gives us a unipotent element in $\XX$. To that end, we choose a random element $\bx{u} \in \KK$ and construct the sequence
\[
\bx{u}^n, \bx{u}^{2n}, \bx{u}^{2^2n}, \ldots, \bx{u}^{2^mn}=1
\]
by using the multiplication in $\KK$. Obviously, if $4$ divides the multiplicative order of $\bx{u}$, then the construction of this sequence for $\bx{u}$ produces a unipotent element; we detect it as a failure in reification of an involution at some step of construction. Because of probabilistic nature of our algorithms, there is a tiny possibility that the resulting element is semisimple -- but we can use our unipotency test, Lemma \ref{lem:test:uni}, to decide whether this element is unipotent or not. If we can not find a unipotent element from this sequence, then we choose another element in $\KK$ randomly and repeat this procedure. Note that at least half of the elements in $\KK$ have multiplicative orders divisible by $4$.

If $\qmone$, then the coset $\TT_{\bx{e}_2} \bx{e}_3$ has $q+1$ involutions. Since 4 does not divide $q-1$ in this case, there exists no element in $\KK$ satisfying $\bx{x}_1^2\oplus \bx{1}=\bx{0}$. Therefore, we should find a solution of the equation $\bx{x}_1^2\oplus \bx{x}_2^2\oplus \bx{1}=\bx{0}$ in $\KK$. To this end, we search for random elements $\bx{x}, \bx{y} \in \KK$ such that $\ominus \bx{x}^2 \ominus \bx{y}^2$ has a square root in $\KK$, which can be checked by using Tonelli-Shanks algorithm. By Lemma \ref{lem:quad:res},  for random $\bx{x}, \bx{y} \in \KK$, $\ominus \bx{x}^2 \ominus \bx{y}^2$ is a non-zero quadratic residue with probability $1/2-1/2q^2$. Assume that $\ominus \bx{x}^2 \ominus \bx{y}^2 = \bx{c}^2$ for some $\bx{c} \in \KK$, and we compute this $\bx{c}\in \KK$ by using Tonelli-Shanks algorithm applied in $\KK$. Then we have $\bx{x}^2\oplus \bx{y}^2\oplus \bx{c}^2=\bx{0}$, or equivalently, $\frac{\bx{x}^2}{\bx{c}^2} \oplus \frac{\bx{y}^2}{\bx{c}^2}\oplus \bx{1}=\bx{0}$. We set $\bx{a}=\bx{x}/ \bx{c}$ and $\bx{b}=\bx{y}/ \bx{c}$. Now, recall that $\bx{b}^{\bx{d}_3}$ lies on the $\bx{x}_2$-axis and its coordinate is the same as the coordinate of the element $\bx{b}$ on the $\bx{x}_1$-axis. Now, the intersection of the lines $\bx{e}_1 \vee \bx{b}^{\bx{d}_3}$ and $\bx{e}_2\vee \bx{a}$ has the homogenous coordinate $(\bx{a}, \bx{b}, \bx{1})$. Clearly, the construction of this intersection point produces a candidate unipotent element in $\XX$. Finally, we use unipotency test, Lemma \ref{lem:test:uni}, to check that it is indeed a unipotent element.

To find the characteristic of the underlying field, we compute the order of the unipotent element that we constructed.
\end{proof}

We  tested our algorithm in GAP for finding unipotent elements in $\so_3(\bF_p)$ for $30$-digit primes like $p=115756986668303657898962467957$: it works.

\section{Coordinatization of the action of $\XX$ on $\Ii$, Proof of Theorem \ref{SL2-SO3}} \label{sec:coordinatisationaction}

The proof of Theorem \ref{SL2-SO3} (a) follows from Theorem \ref{psl2-pgl2} and the proof for part (b), namely, the construction of a black box field $\KK$, follows from the constructions in Section \ref{sec:coordinatisation}.

\subsection{Construction of the morphism $\XX \longrightarrow \so_3(\KK)$}\label{morph:xtoso3}

The aim of this section is to represent the action of an arbitrary element $\bx{x}\in \XX$ on the projective plane $\Pp$  by a $3\times 3$ matrix $\varphi(\bx{x})$ with coefficient in $\KK$. We shall consider several cases:

\textsc{Case 1}. We set
\[
\varphi(\bx{e}_1) = \bbm \bx{1} & \bx{0} &\bx{0} \\ \bx{0} & \ominus\bx{1} &\bx{0} \\ \bx{0} & \bx{0}& \ominus\bx{1} \ebm, \quad
\varphi(\bx{e}_2) = \bbm \ominus\bx{1} & \bx{0} &\bx{0} \\ \bx{0} & \bx{1} &\bx{0} \\ \bx{0} & \bx{0}& \ominus\bx{1} \ebm, \quad
\varphi(\bx{e}_3) = \bbm \ominus\bx{1} & \bx{0} &\bx{0} \\ \bx{0} & \ominus\bx{1} &\bx{0} \\ \bx{0} & \bx{0}& \bx{1} \ebm.
\]

\textsc{Case 2}. Now we compute $\varphi(\bx{u})$ for an arbitrary involution $\bx{u} \in \XX$  in ``general position'' in the sense that $\bx{u}$ does not commute with any of $\bx{e}_i$, $i=1,2,3$.

If $\bx{u} \in \XX$ then involutions $\bx{u}_i = \bx{e}_i^\bx{u}$, $i=1,2,3$, represent the vectors $\epsilon_i^\bx{u}$ in the projective plane $\Pp$.  We can compute the homogeneous  coordinates $(\bx{u}_{i1}, \bx{u}_{i2}, \bx{u}_{i3})$ of $\bx{u}_i$ using construction from Section~\ref{ss:coordinates}. The vector $(\bx{u}_{i1}, \bx{u}_{i2}, \bx{u}_{i3})$ is a scalar multiple of $\epsilon_i^\bx{u}$. We have to normalize it by finding a scalar $\bx{c}_i\in \KK$ such that
\[
\bx{c}_i^2(\bx{u}_{i1}^2 + \bx{u}_{i2}^2 + \bx{u}_{i3}^2) = \mathsf{1}
\]
which is done by taking a square root
\[
\bx{c}_i = \pm \sqrt{\frac{\mathsf{1}}{\bx{u}_{i1}^2 + \bx{u}_{i2}^2 + \bx{u}_{i3}^2}}
\]
(see Section~\ref{ss:squareroot}).
The choice of signs $\pm$ is dictated by the need to make the matrix
\[
U=\left(\bx{u}'_{ij}\right) = \left( \frac{\bx{u}_{ij}}{\bx{c}_i}\right)
\]
an involution from $\so_3(\KK)$; that is, $U$ has to have determinant $1$ and be symmetric.

The choice of signs could happen to be not unique and defined up to simultaneous change of two signs, that is, up to multiplication of $U$ on the right by one of the matrices $\varphi(\bx{e}_i)$. Since $U$ and $\varphi(\bx{e}_i)$ are involutions, their product $U\varphi(\bx{e}_i)$ can happen to be an involution if and only if $U$ and $\varphi(\bx{e}_i)$ commute, which is excluded by our choice of $\bx{u}$.

\textsc{Case 3}. Now let  $\bx{u}\in\XX$ be an involution not in general position, say $\bx{u} \in C_{\XX}(\bx{e}_1)$. Recall that $\CC=C_{\XX}(\bx{e}_1)$ is a dihedral group. If $\bx{u} = \bx{e}_1$, we are in Case 1. If $\bx{u} \ne \bx{e}_1$, we do random search for an involution $\bx{v} \in \XX$ such that $\bx{v}$ and $\bx{w}:=\bx{u}^\bx{v}$ do not commute with any $\bx{e}_1,\bx{e}_2,\bx{e}_3$ (this condition is satisfied with probability $1-O(\frac{1}{q}$)). Then $\bx{u} = \bx{v}\bx{w}\bx{v}$ and we can compute $\varphi(\bx{v})$ and $\varphi(\bx{w})$ as in Case 2 and then compute
\[
\varphi(\bx{u}) = \varphi(\bx{v})\varphi(\bx{w})\varphi(\bx{v}).
\]

\textsc{Case 4}. This is the general case. By Lemma~\ref{lm:bireflectivity}, we know that every $\bx{x}\in\XX$ is either an involution, or  a product of two or three involutions, say $\bx{x} = \bx{u}\bx{v}$; so we compute
\[
\varphi(\bx{x}) = \varphi(\bx{u})\varphi(\bx{v}),
\]
where $\varphi(\bx{u})$ and $\varphi(\bx{v})$ are computed as in Cases 2 and 3.

This gives us an algorithm constructing a morphism $\XX \longrightarrow \so_3(\KK)$.

\subsection{Construction of the  morphism $\so_3(\KK) \longrightarrow \XX$} \label{morph:so3tox}

It is well known that each element $r$ in $\so_3(\KK)$ is either an involution or product of two involutions. In case $r$ is not an involtuion, we can write $r$ as a product of two involutions using Lemma \ref{lm:bireflectivity}. Therefore it will suffice to compute $\varphi^{-1}(r)$ for an involution $r\in\so_3(\KK)$.

We shall think of $r$  as matrix in the same orthonormal basis in which
\[
\varphi(\bx{e}_1) = \bbm \bx{1} & \bx{0} &\bx{0} \\ \bx{0} & \ominus\bx{1} &\bx{0} \\ \bx{0} & \bx{0}& \ominus\bx{1} \ebm, \quad
\varphi(\bx{e}_2) = \bbm \ominus\bx{1} & \bx{0} &\bx{0} \\ \bx{0} & \bx{1} &\bx{0} \\ \bx{0} & \bx{0}& \ominus\bx{1} \ebm, \quad
\varphi(\bx{e}_3) = \bbm \ominus\bx{1} & \bx{0} &\bx{0} \\ \bx{0} & \ominus\bx{1} &\bx{0} \\ \bx{0} & \bx{0}& \bx{1} \ebm.
\]
As it was with computation of $\varphi$, we can easily reduce computation of $\varphi^{-1}(r)$ to the case when $r$ is in general position, that is, $r$ does not commute with any $\varphi(\bx{e}_i)$, $i=1,2,3$.

Being an involution, $r$ is a symmetric matrix; denote its rows as $r_1,r_2,r_3$. Now construct in $\Pp$ points $\bx{s}_i$ which have in the homogeneous coordinates associated with the basis $\bx{e}_1,\bx{e}_2,\bx{e}_3$ the coordinated vectors $r_i$, $i=1,2,3$. The preimage $\bx{s}=\bxg{\varphi}^{-1}(r)$ satisfies the condition
\[
\bx{e}_i^\bx{s} = \bx{s}_i, \quad i=1,2,3.
\]
and is in general position with respect to $\{\bx{e}_i\}$; therefore $\bx{s}$ is uniquely defined by these conditions.

We compute an involution $\bx{t}_1\in \XX$ such that $\bx{e}_1^{\bx{t}_1} = \bx{s}_1$, Lemma \ref{lm:bisection}. Then the element (not necessarily an involution) $\bx{x}=\bx{s}\bx{t}_1$ belongs to $\CC = C_{\XX}(\bx{e}_1)$ and sends $\bx{e}_2$ to $\bx{e}_2^{\bx{x}} = \bx{e}_2^{\bx{s}\bx{t}_1} = \bx{s}_2^{\bx{t}_1} \in\CC$. We solve the conjugation problem once more, this time in $\CC$, and identify this element $\bx{x}\in \CC$; it is defined uniquely up to multiplication by an element from $\EE = \langle \bx{e}_1,\bx{e}_2\rangle$, so we get a coset $\EE \bx{x}$ as an answer. Now $\bx{s} \in \EE \bx{x}\bx{t}_1$, and, being in general position, is the only involution there.

\subsection{Construction of a morphism $\so_3(\bF_q) \rightarrow \so_3(\KK)$}

Let $\bF_q$ be a standard explicitly given finite field of order $q$ and $\bF_p$ be its prime subfield. Assume also that $\KK_0$ is the prime subfield of $\KK$. Then the isomorphism $\bF_p \rightarrow \KK_0$ can be extended to an isomorphism in time polynomial in the input length to an isomorphism $\bF_q \rightarrow \KK$ \cite{maurer07.427}.

\section{Complexities} \label{sec:complexity}

In this section, we compute the complexities of the main procedures presented in this paper. Let $\YY$ be a black box group encrypting $\psl_2(\bF)$ for some finite field $\bF$ of odd characteristic. Let $\mu$ denote  an upper bound on the time requirement for each group operation in $\YY$ and $\xi$ an upper bound on the time requirement, per element, for the construction of random elements of $\YY$. Let $E$ be a global exponent for $\YY$.

In the sequel, we are going to construct a black box group $\XX$ encrypting $\so_3(\bF)$ from the black box group $\YY$. By the construction of this black box group $\XX$, see Subsection \ref{subsec:direct} and Theorem \ref{psl2-pgl2},  the time requirement for each group operation in $\XX$ becomes at most $4\mu$, see Equation (\ref {eq:semi1}) in Subsection \ref{subsec:direct}. An upper bound on the time requirement, per element, for the construction of random elements of $\XX$ is $2\xi$ since we compute in the direct product $\YY \times \YY$.  Moreover, $E$ is replaced by $2E$ when we compute with the black box group $\XX$. For simplicity, we shall denote $E$ as an exponent for both $\YY$ and $\XX$.

We shall express complexities of our procedures in terms of $\mu$, $\xi$ and $E$. We set $E=2^mn$ where $(2,n)=1$.

\subsection{Constructing an involution in $\YY$ (and in $\XX$)} \label{comp:invo}

At least the quarter of elements in $\YY$ and in $\XX$ are of even order \cite[Corollary 5.3]{isaacs95.139}, therefore an involution can be constructed from a random element by repeated square-and-multiply method in time $O(\xi+\mu\log E)$.

\subsection{Centralizer of an involution $\bx{s}$ in $\YY$ (and in $\XX$)} \label{comp:cent}

By the arguments at the end of Subsection \ref{sec:centralizerofprotoinvolution}, a generating set for $\CC_\YY(\bx{s})$ and  $\CC_{\XX}(\bx{s})$ can be constructed in time $O(\xi\log \log E +\mu\log E\log \log E)$.

\subsection{Unipotency test, Lemma \ref{lem:test:uni}} \label{comp:unipotency:test}

For a given involution $\bx{i} \in \XX$ and a random element $\bx{x} \in \XX$, the running time for our test which decides whether the element $\bx{i}\bx{i}^{\bx{x}}$ is unipotent or not is dominated by the complexity of the construction of $\CC_\XX(\bx{i})$. Therefore, the running time for our unipotency test is $O(\xi\log \log E +\mu\log E\log \log E)$.

\subsection{Reification of an involution in $\XX$, Theorems \ref{psl2-pgl2} and \ref{lm:reif}} \label{comp:reif}

We present the complexity for the algorithm in Theorem \ref{lm:reif}. The computation of the complexity for Theorem \ref{psl2-pgl2} is the same.

Given two involutions $\bx{s},\bx{t}\in \XX$, we shall find the complexity of constructing the involution $\bx{j}:=\bx{s}\boxtimes \bx{t}$, if exists, which commutes with both $\bx{s}$ and $\bx{t}$. We set $\bx{z}=\bx{s}\bx{t}$ and check whether $\bx{z}$ has odd or even order which takes time $O(\mu \log E)$.

If $\bx{z}$ has even order then $\bx{j} \in \langle \bx{z} \rangle$ can be computed in time $O(\mu \log E)$, giving the total time $O(\mu\log E)$. If $\bx{z}$ has odd order then, as in Subsection \ref{comp:cent}, $\CC_\XX(\bx{s})$ can be constructed in time $O(\xi\log\log E+\mu\log E\log \log E)$. At this point, we can check whether $\bx{z}$ is unipotent or not, see Subsection \ref{comp:unipotency:test}. Assume that $\bx{z}$ is not unipotent. Note that the elements in the generating set for $\CC_\XX(\bx{s})$ which are not involutions can be taken to be generators for the torus $\TT_\bx{s}$. Let $\SS_{\TT_\bx{s}}$ be a generating set for $\TT_\bx{s}$. By \cite[I.8]{mitrinovic1996}, we can take $|\SS_{\TT_\bx{s}}|=O(\log \log |\bF|)$. Clearly $\SS=\SS_{\TT_\bx{s}}\cup \{\bx{z}\}$ is a generating set for $\XX$ and computing the action of $\bx{j}$ on $\SS$ takes $O(\mu \log \log |\bF|)$ time. Hence, we run the product replacement algorithm on $\SS$ to construct a random element $\bx{x}$ together with its conjugate $\bx{x}^\bx{j}$. Since the elements of the form $\bx{x}^\bx{j}\bx{x}^{-1}$ have odd orders with probability bounded from below by a constant, see \cite{parker10.885}, the construction of $\CC_\XX(\bx{j})$ takes $O(\xi\log\log E+\mu\log E\log \log E)$ time. Finally, the involution $\bx{j}$ can be constructed from an element of even order from the torus in $\CC_\XX(\bx{j})$ by square-and-multiply method. Hence, if $\bx{z}$ has odd order, the construction of the involution $\bx{j}$ is $O(\xi\log\log E + \mu\log E \log \log E)$ time.

\subsection{A line through $\bx{s}$ and $\bx{t}$}\label{comp:line}
 
As this is an application of a reification of an involution, if the involution $\bx{j}:=\bx{s}\boxtimes \bx{t}$ exits then the total time needed to construct $\bx{j}$ is $O(\xi\log \log E+\mu\log E\log \log E)$. If $\bx{j}$ does not exist for these particular involutions $\bx{s}$ and $\bx{t}$ then the reification process returns a unipotent element $\bx{u}$. In this case, we construct a parabolic line $\langle \bx{u}^{\TT_{\bx{s}}}\rangle \bx{s}$  in time $O(\xi\log\log E + \mu\log E \log \log E)$.

\subsection{Intersection of two non-parabolic lines $\kk$ and $\ll$} \label{comp:inter}

Given involutions $\bx{s}_1,\bx{s}_2, \bx{t}_1,\bx{t}_2 \in \XX$, where $\bx{s}_1,\bx{s}_2$ define a line $\kk$ and $\bx{t}_1,\bx{t}_2$ define a line $\ll$, the intersection of $\kk$ and $\ll$, if exists, is the involution  $(\bx{s}_1\boxtimes \bx{s}_2) \boxtimes (\bx{t}_1 \boxtimes \bx{t}_2)$. Therefore, it can be computed in time $O(\xi\log \log E+\mu\log E\log \log E)$, see Subsection \ref{comp:reif}.

\subsection{Tonelli-Shanks algorithm for tori in $\XX$, Lemma \ref{lm:Tonelli-Shanks}} \label{comp:toneshan}

We follow the outline presented in the proof of Lemma \ref{lm:Tonelli-Shanks}.
Let $\TT$ be a cyclic black box group. We use the exponent $E=2^mn$, $n$ odd, for $\TT$. Let $\bx{z}\in \TT$ be an element that has a square root in $\TT$. Checking whether $\bx{z}$ has odd or even order takes $O(\mu \log E)$ time. If $|\bx{z}|$ is odd then the square root of $\bx{z}$, which is $\bx{z}^{(|\bx{z}|+1)/2}$, can be constructed in time $O(\mu \log E)$. If $|\bx{z}|$ is even then we need to look for an element of maximal 2-height in $\TT$. Observe that the proportion of the elements of maximal 2-height in $\TT$ is at least $1/2$ and computing the 2-height of an arbitrary element takes  $O(\mu \log E)$ time. The elements $\bx{a},\bx{b},\bx{c}$ and the corresponding non-negative integer $d$ in the proof of Lemma \ref{lm:Tonelli-Shanks} can be set up in time $O(\mu\log E)$ and $O(\mu m)$, respectively. As the recursion has at most $m$ steps and each step takes at most $O(\mu \log E)$ time, the over all construction takes  $O(\xi + \mu m \log E)$ time.

\subsection{Bisection of angles, Lemma \ref{lm:bisection}} \label{comp:bisect}

Given two conjugate involutions $\bx{i},\bx{j} \in \XX$, we shall find the complexity of constructing a conjugating involution $\bx{x} \in \XX$,  that is, $\bx{i}^\bx{x}=\bx{j}$. The construction of $\bx{x}$ involves finding the square root of $\bx{z}=\bx{i}\bx{j}$ and construction of centralizers of involutions. Therefore it takes $O(\xi \log \log E+ \mu m \log E)$ time, see Subsection \ref{comp:toneshan}. Since $m<\log E$, we have $O(\xi\log \log E + \mu \log^2 E)$.

\subsection{Representation of an arbitrary element as a product of involutions, Lemma \ref{lm:bireflectivity}} \label{comp:proinv}

This is an another application of a reification of an involution, see Subsection \ref{comp:reif}. Hence it takes $O(\xi\log\log E+\mu\log E\log \log E)$ time.

\subsection{Addition and multiplication in $\KK$} \label{comp:aam}

As described in Subsections \ref{sec:add-in-K} and \ref{sec:multiplication-in-K}, addition and multiplication in $\KK$ involve constant number of reifications of involutions. Hence they can be done in time $O(\xi\log\log E +\mu\log E\log \log E)$.

\subsection{Tonelli-Shanks algorithm in $\KK$} \label{comp:toneshaninK}

By Section \ref{comp:aam}, a multiplication in $\KK$ can be done in time $O(\xi\log\log E +\mu\log E\log \log E)$, so following the computations in Section \ref{comp:toneshan}, Tonelli-Shanks algorithm in $\KK$ runs in time $O(\xi m\log E\log\log E+\mu m \log^2E\log\log E)$. Since $m < \log E$,  we have $O(\xi \log^2 E\log\log E+\mu  \log^3E\log\log E)$.

\subsection{Constructing a unipotent element and finding the characteristic, Theorem \ref{cor:unipotent}} \label{comp:uni}

The construction of three commuting involutions in a black box group $\YY$ encrypting $\psl_2(\bF)$ over a field of odd characteristic involves only constructing involutions and their centralizers in $\YY$. Therefore, by Subsection \ref{comp:invo} and \ref{comp:cent}, it can be done in time $O(\xi\log \log E +\mu\log E\log \log E)$. The construction of a black box group $\XX$ encrypting $\so_3$ from the black box group $\YY$ takes $O(\xi\log\log E + \mu\log E \log \log E)$ time by Subsection \ref{comp:reif}. Construction of a black box subgroup encrypting $\Sym_4$ can be done in time $O(\xi\log \log E+\mu\log E\log\log E)$. Hence, we have a set-up for the projective plane  and a black box field $\KK$ in time $O(\xi\log \log E+\mu\log E\log\log E)$.

$\qpone:$ For a random element $\bx{u} \in \KK$, to check whether $\bx{u}$ has a multiplicative order divisible by 4 involves $\log E$ multiplications in $\KK$. Since each multiplication in $\KK$ takes $O(\xi\log\log E +\mu\log E\log \log E)$ time and the orders of the half of the elements in $\KK$ are divisible by 4, the construction of a unipotent element takes at most $O(\xi\log E\log\log E +\mu\log^2 E\log \log E)$ time.

$\qmone:$ By the complexity given in Subsection \ref{comp:toneshaninK}, finding elements $\bx{x}, \bx{y} \in \KK$ where $\ominus \bx{x} \ominus \bx{y}$ is a quadratic residue takes time $O(\xi \log^2 E\log\log E+\mu  \log^3E\log\log E)$. Since the rest of the computation involves addition and multiplication in $\KK$, constructing lines and finding the intersections of lines, the overall construction takes time   $O(\xi \log^2 E\log\log E+\mu  \log^3E\log\log E)$.

\subsection{Morphism $\XX \rightarrow \so_3(\KK)$} \label{comp:xtoso3}

We will find the complexity to represent an involution $\bx{u}\in \XX$ in $\so_3(\KK)$ when $\bx{u}$ does not commute with some commuting right type involutions $\bx{e}_1,\bx{e}_2,\bx{e}_3 \in \XX$.  Then, together with the computations in Subsection \ref{comp:proinv} the complexity of the representation of an arbitrary element follows.

 As in Subsection \ref{comp:uni}, we have a set up for the projective plane and a black box field $\KK$ in time $O(\xi\log \log E+\mu\log E\log\log E)$. Then, the computation of the homogenous coordinates for $\bx{e}_i^\bx{u}=(\bx{u}_{i1},\bx{u}_{i2},\bx{u}_{i3})$ involves finding intersections of the corresponding lines, so it takes $O(\xi\log \log E +\mu\log E\log\log E)$ time. Normalization of $\bx{e}_i^\bx{u}=(\bx{u}_{i1},\bx{u}_{i2},\bx{u}_{i3})$ involves the computation of $\frac{1}{\bx{u}_{i1}^2+\bx{u}_{i2}^2+\bx{u}_{i3}^2}$ and its square root $\bx{c}_i$ in $\KK$. The computation of the quotient $\frac{1}{\bx{u}_{i1}^2+\bx{u}_{i2}^2+\bx{u}_{i3}^2}$ takes  $O(\xi\log\log E+\mu\log E\log \log E)$ time and the computation of square roots $\bx{c}_i$ in $\KK$  takes $O(\xi \log^2 E\log\log E+\mu  \log^3E\log\log E)$ time, see Subsection \ref{comp:toneshaninK}. Hence, the normalization of $\bx{e}_i^\bx{u}$ can be done in time $O(\xi \log^2 E\log\log E+\mu  \log^3E\log\log E)$.  The time needed to compute the matrix $\UU=\left(  \frac{\bx{u}_{ij}}{\bx{c}_i}\right)$ is
$O(\xi \log \log E+\mu\log E\log \log E)$. Hence adding all the complexities above, we get $O(\xi\log^2 E\log\log E+\mu \log^3E\log\log E)$.

\subsection{Morphism $\so_3(\KK) \rightarrow \XX$}

Let $r \in \so_3(\KK)$ be an arbitrary element. To write $r$ as a product of two involutions in $\so_3(\KK)$, we apply the same arguments in Lemma \ref{lm:bireflectivity}. In this case, we take any $s \in \so_3(\KK)$ and look for an involution $t\in \so_3(\KK)$ (a symmetric matrix of order 2) which inverts both $r$ and $s$. Since $r$ and $s$ matrices over black box field $\KK$, to find such an involution $t$, we solve the corresponding systems of linear equations in $\KK$. Hence, this can be done in time $O(\xi\log\log E+\mu\log E\log \log E)$.

Now, let $r\in \so_3(\KK)$ be an involution. As in Subsection \ref{comp:xtoso3}, it is enough to find the complexity for the construction of a black box group element $\bx{r} \in \XX$ encrypting $r$ when  $r$ does not commute with $\varphi(\bx{e}_1), \varphi(\bx{e}_2), \varphi(\bx{e}_3)$.  Let $r_1, r_2, r_3$ be the rows of $r$. Constructing the involutions $\bx{s}_1,\bx{s}_2,\bx{s}_3 \in \Pp$ with the homogenous coordinates $r_1, r_2, r_3$ involve only constant number of reifications of involutions and intersections of lines. Therefore, we can construct these involutions in time  $O(\xi\log \log E+\mu\log E\log \log E)$. Constructing the desired involution $\bx{r} \in \XX$ such that $\bx{e}_i^\bx{r}=\bx{s}_i$ involves two times bisection of angles and the construction of centralizers of involutions in $\XX$ so it takes $O(\xi\log \log E + \mu  \log^2 E\log \log E)$ time by Subsections \ref{comp:cent} and \ref{comp:bisect}. Hence the running time for the construction of the black box group element encrypting $r$ is $O(\xi\log \log E + \mu  \log^2 E\log \log E)$.

\section*{Acknowledgements}

This paper would have never been written if the authors did not enjoy the warm hospitality offered to them at the Nesin Mathematics Village in \c{S}irince, Izmir Province, Turkey, in Summers 2011--17; our thanks go to Ali Nesin and to all volunteers, staff, and students who have made the Village a mathematical paradise.

We thank  Adrien Deloro and Roman Kossak for many fruitful discussions, and Alexander Konovalov and Chris Stephenson who helped us to clarify links with the computer science. 

Our work was partially supported by the Marie Curie FP7 Initial Training Network MALOA (PITN-GA-2008-MALOA no. 238381) and by CoDiMa (CCP in the area of Computational Discrete Mathematics; EPSRC grant EP/M022641/1), and by The Dame Kathleen Ollerenshaw Trust.

In the project, we were using the GAP software package by The GAP Group, GAP--Groups, Algorithms, and Programming, Version 4.8.4; 2016 (\url{http://www.gap-system.org}).

\end{document}